\documentclass[10pt]{siamltex}

\usepackage{graphicx} 
\usepackage{geometry,graphicx,amssymb,amsmath,amsbsy,eucal,amsfonts,mathrsfs,amscd,bm}
\usepackage{color}

\usepackage{tikz-cd}
\usetikzlibrary{arrows,shapes,calc,snakes}
\usetikzlibrary{shapes,snakes}
\usepackage{tikz-3dplot}
\usetikzlibrary{intersections}
\usepackage{tkz-euclide}

\newtheorem{remark}{Remark}[section]

\allowdisplaybreaks[4]

\newcommand{\p}{{\partial}}  

\newcommand{\nab}{\nabla}

\newcommand{\mct}{\mathcal{T}_h}

\newcommand{\Div}{{\rm div}\,}

\newcommand{\bv}{{\bm v}}

\newcommand{\bbR}{\mathbb{R}}

\newcommand{\pol}{\EuScript{P}}
\newcommand{\bpol}{\boldsymbol{\pol}}

\newcommand{\bw}{\bm w}
\newcommand{\bu}{\bm u}
\newcommand{\bn}{\bm n}

\newcommand{\bbN}{\mathbb{N}}
\newcommand{\calT}{\mathcal{T}}

\newcommand{\perim}{|\partial T|}

\newcommand{\bary}{{\rm bary}}
\newcommand{\inc}{{\rm inc}}

\newcommand{\asp}{{\varrho}}

\newenvironment{myproof}[1]{%
 {\em Proof  of #1}.%
}{\endproof}

\title{The Scott-Vogelius method for the Stokes problem on anisotropic meshes}

\author{Kiera Kean\thanks{University of Pittsburgh, Department of Mathematics.  Supported in part by NSF grant DMS-1817542.}\and Michael Neilan\thanks{University of Pittsburgh, Department of Mathematics.  Supported in part by NSF grant DMS-2011733}\and Michael Schneier\thanks{University of Pittsburgh, Department of Mathematics.}}
\date{}
\begin{document}
\maketitle

\begin{abstract}
This paper analyzes the Scott-Vogelius divergence-free element pair on anisotropic meshes. We explore the behavior of the inf-sup stability {constant} with respect to the aspect ratio on meshes generated with a standard barycenter mesh refinement strategy, as well as a newly introduced incenter refinement strategy. 
{Numerical experiments are presented which support the theoretical results.}

\end{abstract}

%%%%%%%%%%%%%%%%%%%%%%%%%%%%%%%%%%%%%%%%%%%%%%%%%%%%
%%%%%%%%%%%%%%%%%%%%%%%%%%%%%%%%%%%%%%%%%%%%%%%%%%%%
\section{Introduction}
Let $\Omega \subset \mathbb{R}^2$ be a regular open {polygon with boundary} $\Gamma$. We consider the Stokes equation with the no-slip boundary condition:
\begin{equation*}
\begin{aligned}
    -\nu \Delta \bu + \nabla p &= {\bm f}  \text{ in } \Omega,    \\
              \nabla \cdot \bu &= 0  \text{ in } \Omega, \\
                            \bu &= 0 \text{ on } \Gamma,
\end{aligned}
\end{equation*}
where $\bu$ is the velocity, $p$ is the pressure, ${\bm f}$ is a given body force, and $\nu$ is the viscosity.

In this manuscript we study the stability of the divergence-free Scott-Vogelius (SV) finite element pair on anisotropic meshes for the Stokes problem;
the results
trivially extend to other divergence free equations, e.g., the incompressible Navier-Stokes equations.
 Divergence-free methods and other pressure-robust schemes are an extremely active field of research
 (cf.~\cite{SIAM_DivFree17,Neilan20}) 
%  with methods ranging from a variety of finite element pairs 
% Another change I didn't really know how to handle, so I tried. Basically 'a variety of finite element pairs' didn't really strike me as a 'method' per se? It seemed gramatically odd but I might be trying to use English grammar for math so 
ranging from a variety of finite element pairs (e.g., \cite{ScottVogelius85,ArnoldQin92,Zhang05,Cockburn_etal05,Austin_etal04,Buffa_etal11,GuzmanNeilan14,FalkNeilan13,Guzman_etal20}) to modifying the formulation of the equations 
 (e.g., \cite{Linke14,Brennecke_etal15,Linke_etal16,LinkeMerdon16,Lederer_etal17,VerfurthZanotti19,KreuzerZanotti20}). Advantages of divergence-free methods include exact enforcement of conservation laws, pressure robustness with the velocity error being independent of the pressure error and viscosity term \cite{SIAM_DivFree17,Ahmed_etal18}, and improved stability and accuracy of timestepping schemes \cite{ART15, DS21}.

The Stokes equation has been studied on anisotropic meshes for a number of different element pairs. In \cite{ANS01} it was shown that for the Crouzeix-Raviart element, the inf-sup constant is independent of the aspect ratio on triangular and tetrahedral meshes. A similar result was shown for the Bernardi-Raugel finite element pair in two dimensions for classes of triangular and quadrilateral meshes in \cite{ApelNicaise04}. Recently, in \cite{BW19} it was shown for a specific class of anisotropic triangulation that the lowest order Taylor-Hood element was uniformly inf-sup stable. A nonconforming pressure robust method was studied in \cite{AKLM21}. Stability and convergence on anisotropic meshes for the Stokes equation has also been studied extensively for the hp-finite element method \cite{ABW15, SS98, SSR99}.

Up to this point there have been no theoretical results for $H^1$ conforming divergence free finite elements on anisotropic meshes. The low-order SV element pair is somewhat unique in that it is not inf-sup stable on general meshes, but requires special meshes e.g., the barycenter refinement {(or Clough-Tocher refinement)} which is obtained by connecting the vertices of each triangle on a given mesh to {its} barycenter. As pointed out in {\cite[p.12]{JNN18}} this gives rise to meshes with possibly very small and large angles. The impact of these angles on the inf-sup constant was stated as an open problem in \cite{JNN18}.

In this work 
%we study the dependency of the inf-sup constant on these angles
we show barycenter refinement on anisotropic meshes will necessarily lead to large angles 
%I feel like commented out section says `we looked at the angles produced by the barycenter refinement and described exactly what these angles do to the inf-sup constant' when our results don't explicitly deal with the angles of the subtriangulations. 
and propose an alternative mesh refinement strategy based on the incenter of each triangle.
This incenter refinement strategy produces a mesh {that} avoids large angles and allows a smaller increase in aspect ratio on refinement. 
We prove there is a linear relationship between the inf-sup constant and the inverse of the aspect ratio for both the barycenter and incenter refined mesh; numerical experiments show that these results are sharp. {Surprisingly, numerical tests indicate that there is not a 
significant difference, in terms of accuracy, between the incenter and barycenter refinement.}
%do not appear to indicate there is a significant difference in terms of accuracy for the incenter versus barycenter refinement. 

%Large portions of this paragraph use the exact same phrasing as the paragraph immediately preceeding it, this seemed jarring to me but may not be actually a problem? Not sure, but I tried to change it, let me know if I should change it back
The rest of this manuscript is organized as follows: In Section \ref{sec-prelims} we introduce notation and give some preliminary results that will be used for the inf-sup stability estimates. We also prove that the incenter refined mesh has superior aspect ratios and angles compared to the barycenter refined mesh. In Section \ref{sec-stability} we prove that the inf-sup constant scales linearly with the inverse of the aspect ratio for both barycenter and incenter refinement. 
In Section \ref{sec-numerics}, we verify numerically the geometric results proven in Section \ref{sec-prelims} and stability results proven in Section \ref{sec-stability}. {We also demonstrate that} there does not appear to be an appreciable difference in terms of accuracy for the incenter versus barycenter refinement. 
Finally, the appendix contains proofs of some technical lemmas.

%%%%%%%%%%%%%%%%%%%%%%%%%%%%%%%%%%%%%%%%%%%%%%%%%%%%%%%%%%%%%%%%%%%%%%%
%%%%%%%%%%%%%%%%%%%%%%%%%%%%%%%%%%%%%%%%%%%%%%%%%%%%%%%%%%%%%%%%%%%%%%%
%%%%%%%%%%%%%%%%%%%%%%%%%%%%%%%%%%%%%%%%%%%%%%%%%%%%%%%%%%%%%%%%%%%%%%%
\section{Preliminaries}\label{sec-prelims}

Let $\mct$ denote a conforming 
simplicial triangulation of $\Omega\subset \bbR^2$.
We denote the vertices and edges of $T$ as $\{z_i\}_{i=1}^3$
and $\{e_i\}_{i=1}^3$ respectively, labeled such that 
$z_i$ is opposite of $e_i$.  
Set $h_i=|e_i|$ and without loss of generality, 
we assume $h_1\le h_2\le h_3$.
We denote by $\rho_T$ the diameter of the incircle of $T$
and set $h_T = h_3$.
Let $\alpha_i$ be the angle of $T$ at vertex $z_i$,
note that $\alpha_1\le \alpha_2\le \alpha_3$.

Let $z_0\in T$ be an interior point of $T$,
and set $T^{ct} = \{K_1,K_2,K_3\}$ to be  the local (Clough-Tocher) triangulation of $T$, obtained
by connecting the vertices of $T$ to $z_0$.
The three triangles $\{K_i\}_{i=1}^3$ are labeled such
that $\p K_i\cap \p T = e_i$.
Let $a_T$ be the altitude of $T$ with respect to edge $e_3$,
and let $k_i$ be the altitude of $K_i$ with respect to $e_i$ (cf.~Figure \ref{fig:CTSingleSplit}).

%%%%%%%%%%%%%%%%%%%%%%%%%%%%%%%%%%%%%%%%%%%%%%%%%%%%%%%%
%%%%%%%%%%%%%%%%%%%%%%%%%%%%%%%%%%%%%%%%%%%%%%%%%%%%%%%%
%%%%%%%%%%%%%%%%%%%%%%%%%%%%%%%%%%%%%%%%%%%%%%%%%%%%%%%%
\subsection{Geometric results and dependence of split point} We examine the {dependencies and properties} of the 
local triangulation of $T$ on the choice of split point $z_0$. 
In particular, we consider geometric properties of the triangulations obtained by connecting vertices of $T$ to the barycenter and the incenter of $T$. 
First, we require a few definitions.
%%%%%%%%%%%%%%%%%
%%%%%%%%%%%%%%%%%
%%%%%%%%%%%%%%%%%
\begin{definition}
The {\em barycenter} of $T$ is given by
\[
z_{\bary} = \frac13 (z_1+z_2+z_3).
\]
The {\em incenter} of $T$ is given by
\[
z_{\inc} = \frac1{|\p T|}(h_1 z_1 + h_2 z_2+h_3 z_3).
\]
\end{definition}

%%%%%%%%%%%%%%%%%
%%%%%%%%%%%%%%%%%
%%%%%%%%%%%%%%%%%
\begin{definition}\
\begin{enumerate}
\item The {\em aspect ratio} of $T$ is given by
\[
\asp_T:=\frac{h_T}{\rho_T} = \frac{|\p T|h_T}{4|T|}.
\]

\item The {\em aspect ratio} of $T^{ct}$, denoted by $\asp_{T^{ct}}$, is defined as the maximum of the aspect ratio of the three triangles in the refinement, i.e.,
\[
\asp_{T^{ct}}:=\max_{K_i\in T^{ct}} \asp_{K_i}.
\]
\end{enumerate}
\end{definition}

%%%%%%%%%%%%%%%%%
%%%%%%%%%%%%%%%%%
%%%%%%%%%%%%%%%%%
\begin{definition}[\cite{AcostaEtAl}]
A triangle $T$ is said to satisfy a large angle condition, written as $LAC(\delta)$, if there exists $\delta >0$ such that $\alpha_i<\pi-\delta$ for $i = 1,2,3$.
%\KK{I was mistaken, it was MAC($\delta$), should I change the L to an M?} 

\end{definition}

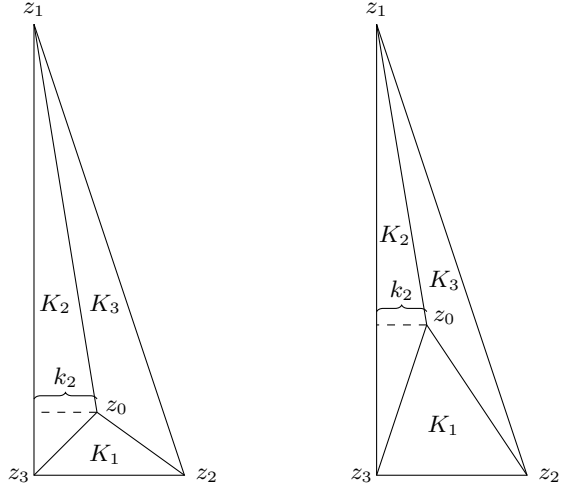
\begin{figure}
\begin{center}
\begin{tikzpicture}[scale=2]
\coordinate (Z3) at (0,0);
\coordinate (Z2) at (1,0);
\coordinate (Z1) at (0,3);

\coordinate (Z0) at (0.41886117,0.41886117);

\coordinate (A2) at (0,0.41886117);

\draw[-](Z1)--(Z2)--(Z3)--(Z1);

\draw[-](Z1)--(Z0);
\draw[-](Z2)--(Z0);
\draw[-](Z3)--(Z0);

\draw[-,dashed](Z0)--(A2);

\draw[decoration={brace,mirror,raise=5pt},decorate]
  (Z0) -- node[above=6pt] {\small $k_2$} (A2);

\node (n3) at (0.55,0.45) {\small$z_0$};
\node (n3) at (-0.1,0.) {\small$z_3$};
\node (n3) at (1.15,0.) {\small$z_2$};
\node (n3) at (0.,3.1) {\small$z_1$};

\node (n3) at (0.472953723,1.13962039) {\small $K_3$};
\node (n3) at (0.13962039,1.13962039) {\small $K_2$};
\node (n3) at (0.472953723,0.13962039) {\small $K_1$};

\end{tikzpicture}
%%%
\qquad\qquad
%%%
\begin{tikzpicture}[scale=2]
\coordinate (Z3) at (0,0);
\coordinate (Z2) at (1,0);
\coordinate (Z1) at (0,3);

\coordinate (Z0) at (1/3,1);

\coordinate (A2) at (0,1.);

\draw[-](Z1)--(Z2)--(Z3)--(Z1);

\draw[-](Z1)--(Z0);
\draw[-](Z2)--(Z0);
\draw[-](Z3)--(Z0);

\draw[-,dashed](Z0)--(A2);

\draw[decoration={brace,mirror,raise=5pt},decorate]
  (Z0) -- node[above=6pt] {\small $k_2$} (A2);

\node (n3) at (0.45,1.05) {\small$z_0$};
\node (n3) at (-0.1,0.) {\small$z_3$};
\node (n3) at (1.15,0.) {\small$z_2$};
\node (n3) at (0.,3.1) {\small$z_1$};

\node (n3) at (0.45,1.2962039) {\small $K_3$};
\node (n3) at (0.11962039,1.5962039) {\small $K_2$};
\node (n3) at (4/9,1/3) {\small $K_1$};
\end{tikzpicture}

%\begin{tikzpicture}
%\draw[->] (0,0) coordinate (O) -- (3,0);
%\draw[->] (O) -- (0,4);
%\node[inner sep=1.5pt,fill,circle,label={60:$(x,y)$}] at (2,3) (point) {};
%\draw (0,0) -- (point);
%\draw (1.8,0) -- ++(0,0.2) -- ++(0.2,0);
%\draw[dashed] (2,0) coordinate (pointx) -- (point); 
%\draw[decoration={brace,mirror,raise=5pt},decorate]
%  (2,0) -- node[right=6pt] {$y$} (point);
%\draw[decoration={brace,mirror,raise=5pt},decorate]
%  (0,0) -- node[below=6pt] {$x$} (2,0);
%\path pic ["$\theta$", draw, ->] {angle=pointx--O--point};  
%\end{tikzpicture}
\end{center}
\caption{\label{fig:CTSingleSplit} The Clough--Tocher split of a single triangle
taking the split point $z_0$ to be the incenter (left) and barycenter (right).}
\end{figure}

%%%%%%%%%%%%%%%%%
%%%%%%%%%%%%%%%%%
%%%%%%%%%%%%%%%%%
\begin{lemma}\label{barycenterLA}
Let the split point be taken to be the barycenter, i.e.,
$z_0 = z_{\bary}$. 
Then as the aspect ratio of $T$ goes to infinity, the largest angle in $K_3$  goes to $\pi$, i.e., the large angle condition will be violated in $T^{ct}$ regardless of the angles of $T$.
\end{lemma}
\begin{proof}
Recall the labeling assumption $h_1\le h_2\le h_3$.
A simple calculation shows that
the side lengths of $K_3$ are $h_3, \frac{\sqrt{2h_2^2+2h_3^2-h_1^2}}{3},\frac{\sqrt{2h_1^2+2h_3^2-h_2^2}}{3}$, and we easily find that each side length is bounded below by $\frac{h_T}{3}$. 

Let $\gamma_1,\gamma_2,\gamma_3$ be the 
angles of $K_3$ at $z_1$, $z_2$, and $z_{\bary}$, respectively.
By properties of the barycenter, $k_3 = \frac13 a_T$, where we recall
that $k_3$ and $a_T$ are, respectively, the altitudes of $K_i$ and $T$  with respect
to $e_3$.  Thus,
\begin{equation*}
    \begin{aligned}
      \sin{\gamma_1} = \frac{3k_3}{\sqrt{2h_2^2+2h_3^2-h_1^2}}\leq \frac{3a_T}{h_T},\\
      \sin{\gamma_2} = \frac{3k_3}{\sqrt{2h_1^2+2h_3^2-h_2^2}}\leq \frac{3a_T}{h_T}.
    \end{aligned}
\end{equation*}
The bound $2h_T\leq\perim\leq 3h_T$ gives us 
\[
\frac{h_T}{a_T} \leq \frac{\perim h_T }{ 4|T|} \leq \frac{3h_T}{2a_T},
\]
and so the aspect ratio of $T$ is equivalent to 
 $\frac{h_T}{a_T}$. Thus, as the aspect ratio of $T$ goes to infinity, $\gamma_1$ and $\gamma_2$ go to zero, implying that $\gamma_3$ goes to $\pi$.\hfill
\end{proof}

%%%%%%%%%%%%%%%%%
%%%%%%%%%%%%%%%%%
%%%%%%%%%%%%%%%%%
\begin{lemma}\label{lem:LACInc}
Let the split point of $T$ be taken to be the incenter, i.e., $z_0 = z_{\inc}$. Then, if $T$ satisfies $LAC(\delta)$, all triangles in $T^{ct}$ satisfy $LAC(\frac{\delta}{2})$.
\end{lemma}
\begin{proof}
The incenter is defined as the intersection of the angle bisectors. Thus, $K_i$ has angles $\frac{\alpha_{i+1}}{2}, \frac{\alpha_{i+2}}{2}, \pi-\frac{1}{2}(\alpha_{i+1}+\alpha_{i+2})$.

As $T$ satisfies $LAC(\delta)$, 
$\alpha_i <\pi - \delta = \alpha_1 + \alpha_2 + \alpha_3 - \delta$, 
and therefore
$\delta \leq \alpha_{i+1}+\alpha_{i+2}$.
We then conclude that $\pi-\frac{1}{2}(\alpha_{i+1}+\alpha_{i+2})\leq \pi-\frac{\delta}{2}$
implying that $K_i$ satisfies $LAC(\frac{\delta}{2})$.
\end{proof}

%%%%%%%%%%%%%%%%%
%%%%%%%%%%%%%%%%%
%%%%%%%%%%%%%%%%%
\begin{lemma}\label{lem:IncAsp}
Let $\asp_{T^{ct}_\inc}$ 
be the aspect ratio of $T^{ct}$ when refined with {respect to} the incenter. The following bounds hold:
 \[
2
\asp_T\leq \asp_{T^{ct}_\inc}
 \leq 2\Big(1+\frac{a_T}{h_T}\Big)\asp_T.
 \]
\end{lemma}
\begin{proof}
First note that if a triangle is refined with the incenter, then the longest edge of each subtriangle is the edge shared with the original triangle.
Indeed,  the angles of the triangle $K_i$ in the refinement are $\frac{\alpha_{i+1}}{2},\frac{\alpha_{i+2}}{2}, \frac{\pi-(\alpha_{i+1}+\alpha_{i+2})}{2}=\frac{\pi+\alpha_i}{2}$.
As $\frac{\pi+\alpha_i}{2}$, the angle at the incenter, is an obtuse angle, it is opposite the longest edge of $K_i,$ the edge shared with $T.$ 
Thus, the aspect ratio of $K_i$ is  $\frac{h_i}{\rho_{K_i}}=\frac{|\partial K_i|h_{i}}{4|K_i|}$.

 By definition of incenter, the altitude of $K_i$ with respect to $e_i$ is the inradius of $T.$ Therefore 
 $|K_i|=\frac{k_i h_i}{2} = \frac{\rho_T h_i}{4} = \frac{|T|h_i}{\perim}$, and so
 %. By remark \ref{hKequalseK}, $h_i=h_{K_i}$ and thus
\begin{equation*}
    \begin{aligned}
\asp_{K_i} & = \frac{|\partial K_i|h_{i}}{4|K_i|}
  =\frac{\perim|\partial K_i|}{4|T|}
  = 
    \frac{|\partial K_i|}{h_T} \asp_T. %\Big(\frac{\rho_T}{ h_T}\Big)    
          \end{aligned}
\end{equation*}
For an arbitrary triangle $K_i$ in the refinement, we have $|\partial K_i| \leq \perim\leq 2h_T +2a_T$, giving us 
\[
\asp_{T^{ct}_\inc}
 \leq 2\Big(1+\frac{a_T}{h_T}\Big)\asp_T.
\]
As $K_3$ shares the longest edge with $T,$ we have $2h_T\leq |\partial K_3|,$ giving us 
\[
2\asp_T\leq\asp_{K_3} \le\asp_{T^{ct}_\inc}.
\]
\hfill
\end{proof}

%%%%%%%%%%%%%%%%%
%%%%%%%%%%%%%%%%%
%%%%%%%%%%%%%%%%%
\begin{lemma}\label{lem:BaryAsp}
Let $\asp_{T^{ct}_\bary}$ be the aspect ratio of $T^{ct}$ when refined with the barycenter. The following bounds hold:
\[
\frac{3}{1+\frac{a_T}{h_T}}\asp_T \leq \asp_{T^{ct}_\bary}  \leq 3\asp_T.
\]
\end{lemma}
\begin{proof}
By properties of the barycenter, $|K_i| = \frac{|T|}{3}$. Thus, the aspect ratio
of $K_i$ is
\begin{equation*}
    \begin{aligned}
\asp_{K_i}
    & = \frac{|\partial K_i|h_{K_i}}{4|K_i|}
    =\frac{3|\partial K_i|h_{K_i}}{4|T|}
    =\frac{3|\partial K_i|h_{K_i}}{\perim h_T}\asp_T. %\Big(\frac{\rho_T}{ h_T}\Big)
          \end{aligned}
\end{equation*}

For all triangles, we have $|\partial K_i|\leq \perim$ and $h_{K_i}\leq h_T$, and so, 
\[
\asp_{T^{ct}_\bary}\leq 3\asp_T . %\Big(\frac{\rho_{T^{ct}}}{h_{T^{ct}}}\Big)_{bary}
\]

For $K_3$ we  use the bounds $ 2h_T\leq|\partial K_3|$ and $ \perim\leq 2h_T + 2a_T$ to get the following lower bound: 
\[
\frac{3}{1+\frac{a_T}{h_T}}\asp_T\leq\asp_{T^{ct}_\bary}.
\]
\hfill
\end{proof}

%%%%%%%%%%%%%%%%%%%%%
%%%%%%%%%%%%%%%%%%%%%
%%%%%%%%%%%%%%%%%%%%%
\begin{remark}
Lemmas \ref{barycenterLA}--\ref{lem:BaryAsp}
indicate superior properties of the incenter refinement
compared to barycenter refinement.  In particular, the incenter
refinement inherits the large angle condition of its parent triangle.
Furthermore, Lemmas \ref{lem:IncAsp}--\ref{lem:BaryAsp} show that
for $T$ with large aspect ratio, the barycenter refinement induces
a triangulation with aspect ratio approximately three times
that of its parent triangle; in contrast, the incenter refinement
yields triangles with aspect ratios approximately twice that of its parent triangle.

On the other hand, we comment that (i) the finite element spaces
given below inherit the approximation properties of the parent triangulation,
in particular, the piecewise polynomial spaces may still possess optimal-order approximation
properties even if $T^{ct}$ does not satisfy the large angle condition; (ii) 
the inf-sup stability constants derived below are given in terms of $\asp_T$ (not $\asp_{T^{ct}}$).
Nonetheless, the analysis will show that, while asymptotically similar
with respect to aspect ratio,
the incenter
refinement leads to better constants in the stability and convergence analysis than the barycenter refinement.
\end{remark}

\begin{remark}
For the rest of the paper, the constant $C$ will
denote a generic positive constant independent of the mesh
size and aspect ratio that may take different values
at each occurence.
\end{remark}

%%%%%%%%%%%%%%%%%%%%%%%%%%%%%%%%%%%%%%%%%%%
%%%%%%%%%%%%%%%%%%%%%%%%%%%%%%%%%%%%%%%%%%%
%%%%%%%%%%%%%%%%%%%%%%%%%%%%%%%%%%%%%%%%%%%
\section{Stability Estimates}\label{sec-stability}
In this section, we derive stability estimates of the lowest-order
Scott-Vogelius Stokes pair in two dimensions.  
This pair is defined on the globally refined Clough-Tocher triangulation
given by
\[
\mct^{ct} = \{K\in T^{ct}: \exists T\in \mct\}.
\]

For a  triangulation $S_h$ and $k\in \bbN_+$, we define the spaces
\begin{alignat*}{2}
\pol_k(S_h) &= \{q\in L^2(D):\ q|_{K}\in \pol_k(K)\ \forall K\in S_h\},\qquad 
&& \mathring{\pol}_k(S_h) = \pol_k(S_h)\cap L^2_0(D),\\
\pol^c_k(S_h) &= \pol_k(T^{ct})\cap H^1(D),\qquad &&\mathring{\pol}_k^c(S_h) = \pol^c_k(S_h)\cap H^1_0(D),
\end{alignat*}
where $D = {\rm int} \bigcup_{K\in S_h} \bar K$.
Analogous vector-valued spaces are denoted in boldface, e.g., $\bpol^c_k(S_h) = [\pol^c_k(S_h)]^2$.
The lowest-order Scott-Vogelius pair is then $\mathring{\bpol}_2^c(\mct^{ct})- \mathring{\pol}_1(\mct^{ct})$.

%We denote by $\hat T$ the reference triangle with vertices $(1,0),(0,1)$ and $(0,0)$,
%and for given $T\in \calT_h$ let $F_T:\hat T\to T$ be an affine bijection
%from $\hat T$ and $T$.  The Jacobian of $F_T$ and its inverse satisfy
%the well-known estimates:
%\begin{equation}\label{eqn:DFBounds}
%|DF_T|\le C h_T,\qquad |DF_T^{-1}|\le C \rho_T^{-1}.
%\end{equation}
%Set $\hat T^{ct}_T$ to be the Clough-Tocher partition 
%of $\hat T$ induced by $F_T$, i.e.,
%\[
%\hat T^{ct}_T = \{\hat K = F_T^{-1}(K):\ K\in T^{ct}\}.
%\]

The  proof of inf-sup stability of the two-dimensional Scott-Vogelius pair
on Clough-Tocher triangulations is based on a macro element technique.
Inf-sup stability is first shown on a single macro element
consisting of three triangles, and then these local results
are ``glued together'' using the stability of the $\bpol^c_2-\pol_0$ pair.

We now summarize the proof
of inf-sup stability of the $\mathring{\bpol}^c_2(\mct^{ct})-\mathring{\pol}_1(\mct^{ct})$ pair
given in \cite[Proposition 6.1]{GuzmanNeilan18}.  The stability proof relies on two preliminary results.  The first states
the well-known stability of the $\bpol_2^c-\pol_0$ pair \cite{BernardiRaugel85,BoffiEtAl}.
The second is a bijective property of the divergence operator
acting on local polynomial spaces \cite{GuzmanNeilan18,ArnoldQin92}.
%%%%%%%%%%%%%%%
%%%%%%%%%%%%%%%
%%%%%%%%%%%%%%%
\begin{lemma}[Stability of $\bpol^c_2-\pol_0$ pair on $\mct$]\label{lem:P2P0}
There exists $\beta_0>0$ such that
\[
\beta_0 \|q\|_{L^2(\Omega)}\le \sup_{0\neq \bv\in \mathring{\bpol}_2^c(\mct)} 
\frac{\int_\Omega (\Div \bv) q}{\|\nab \bv\|_{L^2(\Omega)}}\qquad \forall q\in \mathring{\pol}_0(\mct).
\]
\end{lemma}

%%%%%%%%%%%%%%%
%%%%%%%%%%%%%%%
%%%%%%%%%%%%%%%
\begin{lemma}[Stability on macro element]\label{lem:LocalStab}
Let $T\in \mct$.  Then there exists $\beta_{T^{ct}}>0$
such that for any $q\in \mathring{\pol}_1(T^{ct})$,
there exists a unique $\bv\in \mathring{\bpol}^c_2(T^{ct})$
such that $\Div \bv = q$ and $\|\nab \bv\|_{L^2(T)}\le \beta_{T^{ct}}^{-1} \|q\|_{L^2(T)}$.
\end{lemma}
\begin{remark}
Lemma \ref{lem:LocalStab} implies there
exists $\beta_{T^{ct}}>0$ such that $\|\nab \bv\|_{L^2(T)}\le \beta_{T^{ct}}^{-1} \|\Div \bv\|_{L^2(T)}$ for all $\bv\in \mathring{\bpol}_2^c(T^{ct})$.
For the continuation of the paper, we assume that $\beta_{T^{ct}}$ is the largest
constant such that this inequality is satisfied.
\end{remark}

%%%%%%%%%%%%%%%
%%%%%%%%%%%%%%%
%%%%%%%%%%%%%%%
\begin{theorem}[Stability of SV pair]\label{thm:StdStability}
There holds 
\begin{align}\label{eqn:InfSupStatement}
\beta \|q\|_{L^2(\Omega)} \le \sup_{0\neq \bw\in \mathring{\bpol}^c_2(\mct^{ct})}  \frac{\int_\Omega (\Div \bw)q}{\|\nab \bw\|_{L^2(\Omega)}}\qquad \forall q\in \mathring{\pol}_1(\mct^{ct}),
\end{align}
with 
\[
\beta =\Big((1+\beta_0^{-1}) \beta_*^{-1}  +\beta_0^{-1}\Big)^{-1} = \frac{\beta_0 \beta_*}{\beta_*+\beta_0+1},
\]
where $\beta_0>0$ is given in Lemma \ref{lem:P2P0},
$\beta_* = \min_{T\in \mct} \beta_{T^{ct}}$, and $\beta_{T^{ct}}$ is given in Lemma \ref{lem:LocalStab}.
\end{theorem}
\begin{proof}
Again the proof of this result is found in \cite[Proposition 6.1]{GuzmanNeilan18}.  We provide the proof here for completeness.

For given $q\in \mathring{\pol}_1(\mct^{ct})$, let $\bar{q}\in \mathring{\pol}_0(\mct)$ be its 
$L^2$-projection onto $\mathring{\pol}_0(\mct)$:
\[
\bar{q}|_T = \frac1{|T|} \int_T q\qquad \forall T\in \mct.
\]
Then $(q-\bar q)|_T\in \mathring{\pol}_1(T^{ct})$ for all $T\in \mct$.

By Lemma \ref{lem:LocalStab}, for each $T\in \mct$, there exists $\bv_T\in \mathring{\bpol}^c_2(T^{ct})$
satisfying $\Div \bv_T = (q-\bar q)|_T$ and $\|\nab \bv\|_{L^2(T)}\le \beta_{T^{ct}}^{-1} \|q-\bar q\|_{L^2(T)}$.
We then set $\bv\in \mathring{\bpol}^c_2(\mct^{ct})$ such that $\bv|_T = \bv_T$ for all $T\in \mct$.
Note that $\|\nab \bv\|_{L^2(\Omega)}\le \beta_*^{-1} \|q-\bar q\|_{L^2(\Omega)}$
with $\beta_* = \min_{T\in \mct} \beta_{T^{ct}}$, and therefore
\begin{align*}
\|q-\bar q\|_{L^2(\Omega)}^2 
&= \int_\Omega q(q-\bar q) = \int_\Omega (\Div \bv)q\\
& = \|\nab \bv\|_{L^2(\Omega)}  \frac{\int_\Omega (\Div \bv)q}{\|\nab \bv\|_{L^2(\Omega)}}
 \le \beta_*^{-1} \|q-\bar q\|_{L^2(\Omega)} \sup_{0\neq \bw\in \mathring{\bpol}^c_2(\mct^{ct})}  \frac{\int_\Omega (\Div \bw)q}{\|\nab \bw\|_{L^2(\Omega)}}.
\end{align*}
Thus,
\begin{align}
\label{eqn:Step1}
\|q-\bar q\|_{L^2(\Omega)}
\le \beta_*^{-1} \sup_{0\neq \bw\in \mathring{\bpol}^c_2(\mct^{ct})}  \frac{\int_\Omega (\Div \bw)q}{\|\nab \bw\|_{L^2(\Omega)}}.
\end{align}

We also have, by Lemma \ref{lem:P2P0} and the triangle and Cauchy-Schwarz inequalities,
\begin{align}\label{eqn:Step2}
\beta_0 \|\bar q\|_{L^2(\Omega)}  \le \|q-\bar q\|_{L^2(\Omega)} + \sup_{0\neq \bw\in \mathring{\bpol}^c_2(\mct^{ct})}  \frac{\int_\Omega (\Div \bw)q}{\|\nab \bw\|_{L^2(\Omega)}}.
\end{align}
Combining \eqref{eqn:Step1}--\eqref{eqn:Step2} yields
\begin{align*}
\|q\|_{L^2(\Omega)} 
&\le \|q-\bar q\|_{L^2(\Omega)} +\|\bar q\|_{L^2(\Omega)}\\
&\le (1+\beta_0^{-1}) \|q-\bar q\|_{L^2(\Omega)} + \beta_0^{-1}  \sup_{0\neq \bw\in \mathring{\bpol}^c_2(\mct^{ct})}  \frac{\int_\Omega (\Div \bw)q}{\|\nab \bw\|_{L^2(\Omega)}}\\
&\le \Big((1+\beta_0^{-1}) \beta_*^{-1}  +\beta_0^{-1}\Big)\sup_{0\neq \bw\in \mathring{\bpol}^c_2(\mct^{ct})}  \frac{\int_\Omega (\Div \bw)q}{\|\nab \bw\|_{L^2(\Omega)}}.
\end{align*}
\hfill
\end{proof}

%%%%%%%%%%%%%%%%%%%%%
%%%%%%%%%%%%%%%%%%%%%
%%%%%%%%%%%%%%%%%%%%%
\begin{remark}
The mapping ${\rm div}:\mathring{\bpol}_k^c(T^{ct})\to \mathring{\pol}_{k-1}(T^{ct})$
is surjective for all $k\ge 1$ \cite{GuzmanNeilan18}.  Therefore, the proof of Theorem \ref{thm:StdStability}
easily extends to the $\mathring{\bpol}_k^c(\mct^{ct})-\mathring{\pol}_{k-1}(\mct^{ct})$
pair for $k\ge 2$.
\end{remark}

%%%%%%%%%%%%%%%%%%%%%
%%%%%%%%%%%%%%%%%%%%%
%%%%%%%%%%%%%%%%%%%%%
\begin{remark}
Theorem \ref{thm:StdStability} shows
that the inf-sup constant $\beta$ depends on
inf-sup constants of two related problems: (1) the inf-sup
constant of the $\mathring{\bpol}_2^c(\calT_h)-\mathring{\pol}_0(\calT_h)$ pair $\beta_0$ and (2)
the local inf-sup constant $\beta_{T^{ct}}$ given in Lemma \ref{lem:LocalStab}.
These two stability constants are estimated in subsequent sections.
\end{remark}

%%%%%%%%%%%%%%%%%%%%%%%%%%%%%%%%%%%%%%
%%%%%%%%%%%%%%%%%%%%%%%%%%%%%%%%%%%%%%
\subsection{Estimates of the inf-sup stability constant $\beta_0$
for the $\mathring{\bpol}_2^c(\calT_h)-\mathring{\pol}_0(\calT_h)$ pair}\label{sec:Apel-summary}
We summarize the results in \cite{ApelNicaise04} 
which show that the inf-sup stability constant
$\beta_0$ for the $\mathring{\bpol}_2^c(\calT_h)-\pol_0(\calT_h)$
is uniformly stable (with respect to aspect ratio and mesh size) on a large class of two-dimensional
anisotropic meshes.

We assume that $\calT_h$ is a refinement
of a shape-regular, or isotropic, macrotriangulation $\calT_H$ of triangular or quadrilateral elements with
\[
\bar \Omega = \bigcup_{Q\in \calT_H} \bar Q.
\]
The restriction of the microtriangulation $\calT_h$ to a macroelement
$Q\in \calT_H$ is assumed to be a conforming triangulation of $Q$.  
These triangulations of a macroelement $Q$ (or patch) are classified into three groupings (cf.~\cite[p.92-93]{ApelNicaise04}):
\begin{enumerate}
\item {\bf Patches of isotropic elements:} The triangulation $\calT_h$ restricted to $Q$ 
consists of isotropic elements.
\item {\bf Boundary layer patches:}  All vertices of the triangulation $\calT_h$ restricted to $Q$ are contained in two edges
of $Q$. 
\item {\bf Corner patches:} Two edges with a common vertex are geometrically refined. It is assumed that
$Q$ can be partitioned into a finite number of patches $K$ of isotropic elements or of boundary layer type
such that adjacent patches have the same size.  One hanging node per side is allowed,
but with the restriction that there is an edge $e$ of some $T\in \calT_h$ that joings
the handing node with a node on the opposite side of $K$.
\end{enumerate}

%%%%%%%%%%%%%%%%
%%%%%%%%%%%%%%%%
%%%%%%%%%%%%%%%%
\begin{theorem}[Theorem 1 in \cite{ApelNicaise04}]\label{thm:ApelResult}
Suppose that
 isotropic patches, boundary layer patches, or corner patches are used.
 Then the inf-sup constant $\beta_0$ associated
 with the  $\mathring{\bpol}_2^c(\calT_h)-\mathring{\pol}_0(\calT_h)$
 pair is uniformly bounded from below with respect to the aspect ratio of $\calT_h$.
 \end{theorem}

%%%%%%%%%%%%%%%%%%%%%%%%%%%%%%%%%%%%%%
%%%%%%%%%%%%%%%%%%%%%%%%%%%%%%%%%%%%%%
\subsection{Estimates of inf-sup stability constant $\beta_{T^{ct}}$}

To estimate the local stability constant $\beta_{T^{ct}}$,
we first map $T$ to a ``scaled reference triangle'' (under
the assumption that $T$ satisfies a large angle condition).
The following lemma is a minor modification
of \cite[Theorem 2.2]{AcostaEtAl}.
For completeness, we provide the proof of the result in the appendix.
%%%%%%%%%%%%%%%%
%%%%%%%%%%%%%%%%
%%%%%%%%%%%%%%%%
\begin{lemma}\label{mapisgood}
Let $T$ satisfy $LAC(\delta)$ and have edge lengths $h_1,h_2$, and $h_3$ (with the convention $h_1\le h_2\le h_3$).  Then there exists a triangle $\tilde{T}$ with vertices $\tilde z_3:=(0,0),\tilde z_2:=(h_1,0),\tilde z_1:=(0,h_2)$ that can be mapped to $T$ by an affine bijection $\tilde F_{T}(\tilde x):=A\tilde{x} + b$ where $\|A\|,\|A^{-1}\|\leq C(\delta)$, where $C(\delta)$ depends only on $\delta$, in particular, the constant is independent of the aspect ratio and size of $T$. 
\end{lemma}

Lemma \ref{mapisgood}
implies that
it is sufficient to estimate $\beta_{T^{ct}}$ in the case $T = \tilde T$.
{Indeed, for} given $\bw\in \mathring{\bpol}_2^c(T^{ct})$,
let $\tilde \bw:\tilde T\to \mathbb{R}^2$ be given
via a scaled Piola transform:
\[
\bw(x) = {D\tilde F_T \hat \bw(\tilde x)}\qquad x = \tilde F_T(\tilde x).
\]
We then have $\tilde \bw\in \mathring{\bpol}_2^c(\tilde T_T^{ct})$,
where $\tilde T_T^{ct}$ is the Clough-Tocher partition of $\tilde T$ induced by $\tilde F_T$, i.e.,
\[
\tilde T^{ct}_T = \{\tilde K_i = \tilde F_T^{-1}(K_i):\ K_i\in T^{ct}\}.
\]
By the chain rule, there holds
\begin{align*}
\nab  \bw(x) = D\tilde F_T  \tilde \nab \tilde \bw(\tilde x)D\tilde F_T^{-1},\qquad
\Div \bw(x) = \widetilde \Div \tilde \bw(\tilde x).
\end{align*}
Making a change of variables, and applying Lemma \ref{lem:LocalStab} on $\tilde T^{ct}_T$, we compute
\begin{align*}
\|\tilde \nab \tilde \bw\|_{L^2(\tilde T)}^2 
&\le |\det(D\tilde F_T)| |D\tilde F_T|^2 |D\tilde F_T^{-1}|^2 \|\tilde \nab \tilde \bw\|_{L^2(\tilde T)}^2\\
&\le  \beta_{\tilde T^{ct}_T}^{-2} |\det(D\tilde F_T)| |D\tilde F_T|^2 |D \tilde F_T^{-1}|^2 \|\widetilde \Div \tilde \bw\|_{L^2(\tilde T)}^2\\
&\le  \beta_{\tilde T^{ct}_T}^{-2}  |D\tilde F_T|^2 |D\tilde F_T^{-1}|^2 \| \Div  \bw\|_{L^2( T)}^2.
\end{align*}
Thus, we conclude from Lemma \ref{mapisgood} that
\begin{align*}
\|\nab \bw\|_{L^2(T)}\le  \beta_{\tilde T_T^{ct}}^{-1}   |DF_T| |DF_T^{-1}|  \| \Div  \bw\|_{L^2( T)}\le C \beta_{\tilde T^{ct}_T}^{-1} \|\Div \bw\|_{L^2(T)}.
\end{align*}
The goal of this section then is to estimate $\beta_{\tilde T_T^{ct}}$, i.e.,
to explicitly estimate the stability result stated in Lemma \ref{lem:LocalStab}
in the case $T=\tilde T$.  Of particular interest is the case where the split point $z_0$ is not affine invariant {(e.g., the incenter)},
and therefore standard scaling arguments are not immediately applicable.
To this end, we derive such an estimate by adopting a constructive stability proof 
of the $\mathring{\bpol}_2(\tilde T_T^{ct})-\mathring{\pol}_1(\tilde T_T^{ct})$
pair given in \cite{GuzmanNeilan18}.
The argument is quite involved and requires some additional notation 
and technical lemmas.

First, the mapping $\tilde F_T$ in Lemma \ref{mapisgood}
satisfies $\tilde F_T(\tilde z_i) = z_i$.
Adopting the notation presented in Section \ref{sec-prelims}, we denote the edges
of $\tilde T$ as $\{\tilde e_i\}_{i=1}^3$, labeled such that
$\tilde e_i$ is opposite $\tilde z_i$.
The lengths of the edges of $\tilde T$ are
$\tilde h_1 := h_1 = |\tilde e_1|$, $\tilde h_2 := h_2 = |\tilde e_2|$,
and $\tilde h_3  := |\tilde e_3| = (h_1^2+h^2_2)^{1/2}$.
The labeling assumptions stated
in Section \ref{sec-prelims} implies $\tilde h_1\le \tilde h_2\le \tilde h_3$.

We set $\vec{k} = (\tilde k_2,\tilde k_1)^\intercal = \tilde F_T(z_{0})\in \bbR^2$ 
to be the image of the split point of $T$ onto $\tilde T$.
The notational convention is chosen so that
the altitude of $\tilde K_i$ with respect
to $\tilde e_i$ is $\tilde k_i$ for $i=1,2$.
We also set $\tilde k_3$ to be the altitude of $\tilde K_3$
with respect to $\tilde e_3$.

The main result of this section is summarized in the following theorem. 
 %%%%%%%%%%%%%%%%%%%%%% 
%%%%%%%%%%%%%%%%%%%%%% 
%%%%%%%%%%%%%%%%%%%%%% 
 \begin{theorem}\label{thm:SkewedStability}
 Let  $\tilde \mu\in \mathring{\pol}_1^c(\tilde T^{ct}_T)$ be the hat function 
associated with the split point $(\tilde k_2,\tilde k_1)^\intercal$
and set  $\tilde \asp =\frac{\tilde h_2}{\tilde h_1}$.
%Suppose that $\tilde k_j\le C \tilde h_i$ for all $i,j\in \{1,2,3\}$.
Then there holds, for all $\tilde \bw\in \mathring{\bpol}_2^c(\tilde T^{ct}_T)$,
\[
|\tilde  \bw|_{H^1( \tilde T)}\le C {\tilde \varrho}^{1/2}(1+ |\tilde \mu|_{H^1(\tilde T)}) \|\widetilde{\Div} \tilde \bw\|_{L^2(\tilde T)}.
\]
In particular, there holds $\beta_{\tilde T_{ct}^T} \ge  C \Big(\sqrt{\tilde \asp} (1+|\tilde \mu|_{H^1(\tilde T)})\Big)^{-1}$.
\end{theorem}

To prove Theorem  \ref{thm:SkewedStability} we require
two scaling results 
whose proofs are given in the appendix.
%%%%%%%%%%%%%%%%%%%%
%%%%%%%%%%%%%%%%%%%%
%%%%%%%%%%%%%%%%%%%%
\begin{lemma}\label{lem:BDMScaling}
Set $\tilde \asp = \frac{\tilde h_2}{\tilde h_1}$.
Then for  $\tilde \bv\in \bpol_1(\tilde T)$,
there holds
\begin{align*}
\| \nab  \tilde \bv\|_{L^2(\tilde T)}^2 +\|\tilde \bv\|_{L^\infty(\tilde T)}^2 &\le C\tilde \varrho^{} |\tilde T|^{-1}  \sum_{i=1}^3 \tilde h_i \| \tilde \bv \cdot  \tilde\bn_i\|_{L^2(\tilde e_i)}^2.
%
%\| \bv\|_{L^\infty(\tilde T)}^2  
%&\le  C  |\tilde T|^{-1} \varrho \sum_{i=1}^3 h_i \| \bv \cdot  \bn_i\|_{L^2( e_i)}^2.
\end{align*}
\end{lemma}

 %%%%%%%%%%%%%%%%%%%%%% 
%%%%%%%%%%%%%%%%%%%%%% 
%%%%%%%%%%%%%%%%%%%%%% 
 \begin{lemma}\label{lem:Tracey}
 For any $\tilde q\in \pol_1(\tilde K_i)$, there holds
 \[
 \tilde k_i \|\tilde q\|_{L^2(\tilde e_i)}^2 \le C \|\tilde q\|_{L^2(\tilde K_i)}^2.
 \]
 \end{lemma}
 
\begin{myproof}{Theorem \ref{thm:SkewedStability}}

The main idea of the proof 
is to write $ \tilde \bw =  \tilde \mu  \tilde \bw_1+ \tilde \mu^2  \tilde \bw_0$,
where $ \tilde \bw_j\in \bpol_j(\tilde  T)$ are specified
by  Brezzi-Douglas-Marini degrees of freedom (DOFs).
This decomposition of $\tilde  \bw$ is unique.

{\em Step 1: Construction of $ \tilde \bw_1$:}\\
Set $ \tilde q :=  \widetilde{\Div}  \tilde \bw \in \mathring{\pol}_1(\tilde T^{ct})$, and define $ \tilde \bw_1\in \bpol_1( \tilde T)$
uniquely by the DOFs
\[
 \int_{ \tilde e_i} (\tilde \bw_1\cdot  \tilde \bn_i) \tilde \kappa = - \tilde k_i \int_{ \tilde e_i }   \tilde q  \tilde \kappa \qquad \forall  \tilde \kappa \in \pol_1( \tilde e_i),\ i=1,2,3.
\]
Thus, $ \tilde \bw_1\cdot  \tilde \bn_i|_{\tilde e_i} = -  \tilde k_i  \tilde q|_{ \tilde e_i}$,
and therefore, since $\tilde \nab \tilde \mu|_{\tilde K_i} = -\tilde k_i^{-1}\tilde \bn_i$ ($i=1,2,3$),
\begin{equation}\label{eqn:hatW1}
 \tilde \bw_1\cdot  \tilde \nab  \tilde\mu|_{\p \tilde  T} =  \tilde q|_{\p \tilde T}.
\end{equation}

{\em Step 2: Construction of $\bw_0$:}\\
Set 
\begin{equation}\label{eqn:q0Def}
\tilde  q_0 = \frac{-1}{ \tilde \mu} \big( \widetilde{\Div} (\tilde \mu  \tilde \bw_1) -  \tilde q\big).
\end{equation}
By \eqref{eqn:hatW1}, $( \widetilde{\Div} ( \tilde \mu  \tilde \bw_1)- \tilde q)|_{\p \tilde  T} = ( \tilde \nab  \tilde \mu \cdot  \tilde \bw_1 -  \tilde q)|_{\p  \tilde T} = 0$,
and therefore we conclude $ \tilde q_0\in \pol_0(\tilde T^{ct})$.
We also have
\[
\int_{\tilde T}  \tilde \mu  \tilde q_0
  = -\int_{\tilde T} ( \widetilde{\Div} ( \tilde \mu \tilde \bw_1) -  \tilde q) = 0.
\]

Let $ \tilde \bw_0\in \bpol_0(\tilde T)$
be uniquely determined by
\begin{equation}\label{eqn:w0Cond}
{2}\int_{\tilde e_i} (\tilde \bw_0\cdot  \tilde \bn_i) = -  \tilde k_i \int_{\tilde e_i}  \tilde q_0  \qquad  i=1,2,
\end{equation}
i.e., $2  \tilde \bw_0 \cdot  \tilde\nab \tilde \mu|_{ \tilde e_i} =  \tilde q_0|_{\tilde e_i}\ (i=1,2)$, 
which implies $2  \tilde \bw_0 \cdot  \tilde \nab  \tilde\mu|_{\tilde K_i} =  \tilde q_0|_{\tilde K_i}\ (i=1,2)$
because all of the functions in the expression are piecewise constant.
We then calculate
\[
 \Div (\tilde  \mu^2 \tilde  \bw_0)|_{\tilde K_i} = 2  \tilde \mu (\tilde \bw_0 \cdot  \tilde\nab  \tilde\mu)|_{ \tilde K_i} =  \tilde \mu  \tilde q_0|_{\tilde K_i}\qquad i=1,2,
\]
and so,
\[
\int_{ \tilde K_3}  \tilde \mu (2 \tilde \nab \tilde \mu \cdot \tilde \bw_0 - \tilde q_0) = \int_{ \tilde K_3} ( \widetilde{\Div} ( \tilde \mu^2  \tilde \bw_0) - \tilde \mu  \tilde q_0 )
=\int_{\tilde T} ( \widetilde{\Div} ( \tilde \mu^2  \tilde \bw_0) -  \tilde \mu \tilde  q_0 ) = 0.
\]
Thus, $2 \tilde \bw_0 \cdot  \tilde \nab  \tilde \mu|_{\tilde K_3} = \tilde q_0|_{\tilde K_3}$, and we
conclude 
\[
\widetilde{\Div} ( \tilde \mu^2 \tilde \bw_0) = \tilde  \mu  \tilde q_0 = -(\widetilde{\Div} ( \tilde \mu  \tilde \bw_1) -  \tilde q\big) \text{ in $\tilde T$},
\]
that is,
\[
\widetilde{\Div} ( \tilde \mu \tilde  \bw_1)+ \widetilde{\Div} ( \tilde \mu^2 \tilde \bw_0)  = \tilde q.
\]
Finally, we set $\tilde \bw =  \tilde \mu \tilde \bw_1 +  \tilde \mu^2  \tilde \bw_0\in \mathring{\bpol}_2^c( \tilde T^{ct})$,
so that $ \widetilde{\Div}  \bw =  q$.  

{\em Step 3: Estimate of $|\tilde \bw|_{H^1(\tilde T)}$:}

We estimate norms of $\tilde \bw_1$ and $\tilde \bw_0$
separately to derive an estimate of $|\tilde \bw|_{H^1(\tilde T)}$.  First, recall that $\tilde \bw_1\cdot \tilde \bn_i|_{\tilde e_i} = \tilde k_i \tilde q|_{\tilde e_i}$,
and therefore by Lemmas \ref{lem:BDMScaling}
and \ref{lem:Tracey},
\begin{align}\label{eqn:w1Boundy}
\|\tilde \bw_1\|_{L^\infty(\tilde T)}^2 +\|\tilde \nab\tilde  \bw_1\|_{L^2(\tilde T)}^2
&\le C \tilde \varrho^{} |\tilde T|^{-1} \sum_{i=1}^3\tilde  h_i \tilde  k_i^2 \|\tilde q\|_{L^2(\tilde e_i)}^2\\
&\nonumber \le C \tilde \varrho^{}|\tilde T|^{-1}  \sum_{i=1}^3 \tilde h_i \tilde k_i \|\tilde q\|_{L^2(K_i)}^2\le C \tilde \varrho^{} \|\tilde q\|_{L^2(\tilde T)}^2.
\end{align}

To estimate $\tilde \bw_0$, we use a more explicit calculation.
To this end, let $\{\tilde \lambda_j\}_{j=1}^3 \subset \bpol_1(\tilde T)$ be the barycenter
coordinates of $\tilde T$, labeled such that $\tilde \lambda_j(\tilde z_i) = \delta_{i,j}$.
We then write
\begin{equation}\label{eqn:qExpansion}
 \tilde q|_{\tilde K_i} =  \sum_{j=1}^3 a_{i,j}\tilde \lambda_j\qquad a_{i,j}\in \mathbb{R}.
\end{equation}
Note that $a_{i,j} = \tilde q|_{\tilde K_i} (\tilde z_j)$ for $i\neq j$.

A calculation then shows (cf.~\eqref{eqn:hatW1})
\begin{align*}
\tilde \bw_1 = 
\begin{pmatrix}
\tilde k_2 \tilde q|_{\tilde K_2} + \tilde \lambda_2 c_2\\
\tilde k_1 \tilde q|_{\tilde K_1} + \tilde \lambda_1 c_1
\end{pmatrix},
\end{align*}
where the constants $c_j\in \bbR$ are given by
\begin{align}\label{eqn:PITA}
c_j = \frac{-1}{\tilde h_j} \sum_{i=1}^3 a_{i,j} \tilde h_i \tilde k_i =  \frac{-2}{\tilde h_j} \sum_{i=1}^3 |\tilde K_i| a_{i,j}.
\end{align}
Another calculation shows that (cf.~\eqref{eqn:q0Def})
\begin{align*}
\tilde q_0 |_{K_1} &= -(\widetilde \Div \tilde \bw_1 + \frac{c_1}{\tilde h_2}),\\
\tilde q_0 |_{K_2} &= -(\widetilde \Div \tilde \bw_1 + \frac{c_2}{\tilde h_1}),
\end{align*}
and therefore (cf.~\eqref{eqn:w0Cond})
\begin{align}\label{eqn:bw0Exp}
\tilde \bw_0 = -\frac12
\begin{pmatrix}
\tilde k_1 (\widetilde{\Div} \tilde \bw_1+\frac{c_1}{\tilde h_2})\\
\tilde k_2 (\widetilde{\Div} \tilde \bw_1 + \frac{c_2}{\tilde h_1})
\end{pmatrix}.
\end{align}

Because $\widetilde\Div \tilde \bw_1$ is constant, %, and using the assumption $\tilde k_i\le C \tilde h_j$, 
we have 
%\begin{align}\label{eqn:Divbw1Bound}
%\tilde k_i \|\widetilde \Div \tilde \bw_1\|_{L^\infty(\tilde T)} = \tilde k_i |\tilde T|^{-1/2} \|\widetilde{\Div} \tilde \bw_1\|_{L^2(\tilde T)}\le C |\tilde \bw_1|_{H^1(\tilde T)}\quad i=1,2.
%\end{align}
\begin{align*}
|\tilde T| |\widetilde \Div \tilde \bw_1|^2  
&= \int_{\tilde T} |\Div \tilde \bw_1|^2
 = (\widetilde \Div \tilde\bw_1) \int_{\p \tilde T} \tilde\bw_1 \cdot \tilde\bn
 = -(\widetilde \Div \tilde \bw_1) \sum_{i=1}^3 \int_{\tilde e_i} \tilde k_i \tilde q.
\end{align*}
Therefore by the Cauchy-Schwarz inequality and Lemma \ref{lem:Tracey}, we have
\begin{align*}
|\widetilde \Div \tilde \bw_1|
&\le |\tilde T|^{-1} \sum_{i=1}^3 \tilde h_i^{1/2} \tilde k_i \|\tilde q\|_{L^2(\tilde e_i)}\\
&\le |\tilde T|^{-1} \sum_{i=1}^3 \tilde h_i^{1/2} \tilde k^{1/2}_i \|\tilde q\|_{L^2(\tilde K_i)}\\
&\le C |\tilde T|^{-1} \sum_{i=1}^3  |\tilde K_i|^{1/2}\|\tilde q\|_{L^2(\tilde K_i)}\le C |\tilde T|^{-1/2} \|\tilde q\|_{L^2(\tilde T)}.
\end{align*}
Noting that $\tilde k_2\le \tilde h_1$ and $\tilde k_1\le \tilde h_2$ by definition of $\tilde T$, we find
\begin{align}\label{eqn:Divbw1Bound}
\tilde k_i |\widetilde \Div \tilde \bw_1|\le \tilde \asp^{1/2} \|\tilde q\|_{L^2(\tilde T)}.
\end{align}

We now show $\frac{\tilde k_1 |c_1|}{\tilde h_2}\le C\|\tilde q\|_{L^2(\tilde T)}$, 
where the constant $c_1$ is given by \eqref{eqn:PITA}.
We first note that, by \eqref{eqn:qExpansion}, for $i\neq j$, 
\begin{align*}
|\tilde K_i| |a_{i,j}| = |\tilde K_i| |\tilde q|_{K_i} (a_j)|\le |K_i| \|\tilde q\|_{L^\infty(K_i)}\le C |\tilde K_i|^{1/2} \|\tilde q\|_{L^2(\tilde K_i)},
\end{align*}
where a standard scaling argument (inverse estimate) was used in the last inequality.
On the other hand, the value of $q|_{K_1}$ at the split point $\vec{k} = (\tilde k_2,\tilde k_1)^\intercal $ is
\begin{align*}
\tilde q|_{K_1}(\vec{k}) = 
a_{1,1} \tilde \lambda_1(\vec{k}) + \tilde q|_{\tilde K_1} (\tilde z_2) \tilde \lambda_2 (\vec{k}) + \tilde q|_{\tilde K_1}(\tilde z_3)\tilde \lambda_3(\vec{k}).
\end{align*}
Using $\tilde \lambda_1(\vec{k}) = \frac{\tilde k_1}{\tilde h_2}$ and $0\le \tilde \lambda_j\le 1$, we conclude
$|a_{1,1}|\le C \frac{\tilde h_2}{\tilde k_1}\|\tilde q\|_{L^\infty(\tilde K_1)}$.  Therefore,
\begin{align*}
|\tilde K_1| |a_{1,1}| \le C |K_1| \frac{\tilde h_2}{\tilde k_1} \|\tilde q\|_{L^\infty(\tilde K_1)}\le C |\tilde K_1|^{1/2} \frac{\tilde h_2}{\tilde k_1} \|\tilde q\|_{L^2(\tilde K_1)}.
\end{align*}

Thus, using $\tilde k_2\le \tilde h_1$ and $\tilde k_1\le \tilde h_2$, we have
%$\tilde k_i\le C \tilde h_j$,
\begin{align*}
\frac{\tilde k_1 |c_1|}{\tilde h_2}
&= \frac{2 \tilde k_1 }{\tilde h_1 \tilde h_2} \Big||\tilde K_1| a_{1,1} + |\tilde K_2| a_{2,1} + |\tilde K_3| a_{3,1}\Big|\\
&\le \frac{C \tilde k_1 }{\tilde h_1 \tilde h_2} \big(|\tilde K_1|^{1/2} \frac{\tilde h_2}{\tilde k_1} + |\tilde K_2|^{1/2}+ |\tilde K_3|^{1/2}\big) \|\tilde q\|_{L^2(\tilde \tilde T)}\le C \tilde \asp^{1/2} \|\tilde q\|_{L^2(\tilde T)}.
%
%&\le C\Big( (k_1 h_1)^{1/2} \frac{1}{h_1} + \frac{k_1}{h_1 h_2} (k_2 h_2)^{1/2}+  \frac{k_1}{h_1 h_2} (k_3 h_3)^{1/2}\Big)\|q\|_{L^2(\tilde T)}\\
%&\le C\|\tilde q\|_{L^2(\tilde T)}.
\end{align*}
The same arguments show
\begin{align*}
\frac{\tilde k_2 |c_2|}{\tilde h_1}\le C \tilde \asp^{1/2}\|\tilde q\|_{L^2(\tilde T)}.
\end{align*}

Thus, we conclude from \eqref{eqn:bw0Exp} and \eqref{eqn:Divbw1Bound}, that
\begin{align}\label{eqn:bw0Estimate}
\|\tilde \bw_0\|_{L^\infty(\tilde T)}
%&\le C\big( |\tilde \bw_1 |_{H^1(\tilde T)}+ \|\tilde q\|_{L^2(\tilde T)}\big)
\le C {\tilde \varrho}^{1/2} \|\tilde q\|_{L^2(\tilde T)}.
\end{align}

Finally, we combine \eqref{eqn:w1Boundy} and \eqref{eqn:bw0Estimate} to obtain
\begin{align*}
|\tilde \bw|_{H^1(\tilde T)}
&\le C\big( \|\tilde  \mu\|_{L^\infty(\tilde T)} |\tilde  \bw_1|_{H^1(\tilde T)}
+ |\tilde  \mu|_{H^1(\tilde T)} \|\tilde  \bw_1\|_{L^\infty(\tilde T)}+ |\tilde  \mu|_{H^1(\tilde T)} \|\tilde  \bw_0\|_{L^\infty(\tilde T)}\big)\\
&\le C \tilde  \varrho^{1/2} (1+ |\tilde  \mu|_{H^1(\tilde T)}) \|\tilde  q\|_{L^2(\tilde T)}.
\end{align*}\hfill
\end{myproof}

%%%%%%%%%%%%%%%%%%%%
%%%%%%%%%%%%%%%%%%%%
%%%%%%%%%%%%%%%%%%%%
\begin{corollary}\label{lem:HatScale}
There holds
%Let $\tilde \mu\in \mathring{\pol}_1^c(\tilde T^{ct})$ be the hat function 
%associated with the split point $(\tilde k_2,\tilde k_1)^\intercal $.  Then
\begin{align*}
|\tilde \mu|_{H^1(\tilde T)}^2 = \frac12 \sum_{i=1}^3 \frac{\tilde h_i}{\tilde k_i},
\end{align*}
and therefore, under the assumptions stated in Theorem \ref{thm:SkewedStability},
\[
|\tilde  \bw|_{H^1( \tilde T)}\le C {\tilde \varrho}^{1/2}\left(\sum_{i=1}^3 \frac{\tilde h_i}{\tilde k_i}\right)^{1/2} \|\widetilde{\Div} \tilde \bw\|_{L^2(\tilde T)}\qquad \forall \tilde \bw\in \mathring{\bpol}_2^c(\tilde T_T^{ct}).
\]
\end{corollary}
\begin{proof}
The function $\tilde \mu$ satisfies  $\tilde \nab \tilde \mu|_{\tilde K_i} = -\frac1{\tilde k_i} \tilde \bn_i$.  Therefore,
\begin{align*}
|\tilde \mu|_{H^1(\tilde T)}^2
& = \sum_{i=1}^3 |\tilde K_i| \tilde k_i^{-2}
 = \frac12 \sum_{i=1}^3 \frac{\tilde h_i}{\tilde k_i}.
 \end{align*}\hfill
 \end{proof}

We now apply Corollary \ref{lem:HatScale}
to two situations, each determined by the location
of the split point of $T^{ct}$: barycenter refinement
and incenter refinement.

\subsubsection{Estimates of inf-sup stability constant $\beta_{T^{ct}}$
on barycenter refined meshes}
The barycenter of a triangle is preserved 
via affine diffeomorphisms and therefore, if the split point is taken to be the barycenter ($z_0 = z_{\bary}$),
 the local triangulation on the reference triangle $\tilde T_{T}^{ct}$
is independent of $T$.  In this setting $(\tilde k_2,\tilde k_1)^\intercal = \frac13(\tilde h_1,\tilde h_2)^\intercal$
is the barycenter of $\tilde T$, and $\tilde k_3 = \frac{\tilde h_1 \tilde h_2}{\tilde h_3}$.
Thus, we have
\begin{align*}
|\tilde \mu|_{H^1(\tilde T)}^2 = \frac{1}{2}\sum_{i=1}^3 \frac{\tilde h_i}{\tilde k_i}
& = \frac{1}{2}\Big(\frac{\tilde h_1}{\tilde h_2} + \frac{\tilde h_2}{\tilde h_1} + \frac{\tilde h_3^2}{\tilde h_1 \tilde h_2} \Big)\le \frac{3}{2} \tilde \varrho.
\end{align*}
%\textcolor{magenta}{I think the 1/2 was dropped I put it back in}
Via Theorem \ref{thm:SkewedStability} and mapping back to $T$, 
we have a refinement of 
Lemma \ref{lem:LocalStab} on barycenter refined meshes. 
%%%%%%%%%%%%%%%%%%%%
%%%%%%%%%%%%%%%%%%%%
%%%%%%%%%%%%%%%%%%%%
\begin{lemma}\label{lem:StabBary}
Suppose that the split point of $T^{ct}$ is the barycenter of $T$, and that $T$ 
satisfies the large angle condition.
Then
\[
\|\nab \bv\|_{L^2(T)}\le C \varrho_T^{} \|\Div \bv\|_{L^2(T)}\qquad \forall \bv\in \mathring{\bpol}_2^c(T^{ct}).
\]
\end{lemma}

%%%%%%%%%%%%%%%%%%%%%%%%%%%%%%%%%%%%%%%%%%%%%%%%%%
%%%%%%%%%%%%%%%%%%%%%%%%%%%%%%%%%%%%%%%%%%%%%%%%%%
\subsubsection{Estimates of inf-sup stability constant $\beta_{T^{ct}}$
on incenter refined meshes}

The incenter of $T$ is $z_{\inc} = \frac{1}{\perim}(h_1 z_1+h_2z_2 + h_3z_3).$  Using the affine transformation given in Lemma \ref{mapisgood},
we have $(\tilde k_2,\tilde k_1)^\intercal=A^{-1}(z_{\inc}-b)$.  Using the formula for $A$ in the proof of Lemma \ref{mapisgood}, we calculate
\begin{equation*}
    \begin{aligned}
(\tilde k_2,\tilde k_1)^\intercal %&=  A^{-1}(z_{in,T} -z_3)\\
&=
    A^{-1}\Bigg(\frac{1}{\perim}(h_1 z_1+h_2z_2 + h_3z_3) -\frac{(h_1+h_2+h_3)z_3}{\perim}\Bigg)
    \\&=
    A^{-1}\Big( \frac{h_1(z_1-z_3)+h_2(z_2-z_3)}{\perim}  \Big)\\
    %&=
    %\frac{h_1}{\perim}A^{-1}(z_1-z_3)+\frac{h_2}{\perim}A^{-1}(z_2-z_3)
        %\\
        &=
    \frac{h_1}{\perim}\tilde{z_1}+\frac{h_2}{\perim}\tilde{z_2}
    %\\&=
%    \frac{h_1}{\perim}(0,h_2)+ \frac{h_1}{\perim}(h_1,0)
    =
    \frac{h_1h_2}{\perim}(1,1)^\intercal.
    \end{aligned}
\end{equation*}
Thus, $\tilde k_1=\tilde k_2 = \frac{h_1h_2}{\perim}$, %= \frac12 h_3 \asp_T$,
and 
\begin{align*}
{\tilde k_3} 
&= \frac{h_1 h_2 - h_2 \tilde k_2 - h_1 \tilde k_1}{\tilde h_3} = \frac{h_3}{\tilde h_3} \frac{h_1 h_2}{|\p T|}.
\end{align*}

We then compute, via Corollary \ref{lem:HatScale},
\begin{align*}
|\tilde \mu|_{H^1(\tilde T)}^2 = \frac12 \sum_{i=1}^3 \frac{\tilde h_i}{\tilde k_i} 
= \frac{|\p T|}{ 2 h_1 h_2} \big(\tilde h_1 + \tilde h_2 + \frac{\tilde h_3^2}{h_3}\big)
{\le\frac{\perim}{4|T|}(h_3+h_3+2h_3)= 4\asp^{}_T.}
\end{align*} 

% \begin{align*}
% |\tilde \mu|_{H^1(\tilde T)}^2 
% = 
% \frac12 \sum_{i=1}^3 \frac{\tilde h_i}{\tilde k_i} 
% = 
% \frac{|\p T|}{ 2 h_1 h_2} \big(\tilde h_1 + \tilde h_2 + \frac{\tilde h_3^2}{h_3}\big)
% \KK{\le\frac{|\partial\tilde{T}|}{4|\tilde{T}|}\frac{\perim}{|\partial\tilde{T}|}(\tilde{h}_3+\tilde{h}_3+\sqrt{2}\tilde{h}_3)= 4\tilde{\asp}^{-1}.}
% \end{align*} 

%\MJN{\bf [Add details to last inequality.]}
%\KK{Not sure how many details need to be added, so I added one extra line, full details are:
%
%$2h_1h_2 \ge 4|T|$
%
%$\tilde{h_1},\tilde{h_2}=h_1,h_2 \leq h_3$
%
%$\frac{\tilde{h_3^2}}{h_3}=h_3\frac{\tilde{h_3^2}}{h_3^2}$
%
%$\frac{\tilde{h_3}}{h_3} \leq \frac{\sqrt{h_1^2+h_2^2}}{h_2} = \sqrt{\frac{h_1^2}{h_2^2}+1} \leq \sqrt{2}$ (given that $h_1\leq h_2$)
%
%Unsure if we just leave it at 4 and trust people to understand it or add more steps? (how much details is details?)

% Also not sure why we do it in terms of $\varrho_T$ not $\tilde{\varrho}$ so I included that below (it seems more general to me?)

% Full details:
% $2h_1h_2 = 4|\tilde{T}|$

% multiply by $\frac{|\partial \tilde{T}|}{|\partial \tilde{T}|}$

% $\tilde{h_1},\tilde{h_2}\leq \tilde{h}_3, \frac{\tilde{h_3}}{h_3}\leq \sqrt{2} \rightarrow \tilde h_1 + \tilde h_2 + \frac{\tilde h_3^2}{h_3} \leq (2+\sqrt{2})\tilde{h}_3$ 

% $\frac{\perim}{|\partial \tilde{T}|} \leq \frac{2h_1+2h_2}{h_1+h_2+\sqrt{h_1^2+h_2^2}}\leq\frac{4}{2+\sqrt{2}}$

%}

%%%%%%%%%%%%%%%%%%%%
%%%%%%%%%%%%%%%%%%%%
%%%%%%%%%%%%%%%%%%%%
\begin{lemma}\label{lem:StabInc}
Suppose that the split point of $T^{ct}$ is the incenter of $T$ and that $T$ 
satisfies the large angle condition.
Then
\[
\|\nab \bv\|_{L^2(T)}\le C \varrho_T^{} \|\Div \bv\|_{L^2(T)}\qquad \forall \bv\in \mathring{\bpol}_2^c(T^{ct}).
\]
\end{lemma}

%\MJN{\bf [Add remark on optimal split point?, i.e., one that minimizes $|\tilde \mu|_{H^1(\tilde T)}$, e.g., ]
%\begin{remark}
%
%
%
%We compute 
%\[
%2 |\tilde \mu|^2_{H^1(\tilde T)}  = \frac{\tilde h_1}{\tilde k_1}+ \frac{\tilde h_2}{\tilde k_2} + \frac{\tilde h_3}{\tilde k_3}.
%\]
%Write
%\[
%\tilde k_3 = \frac{\tilde h_1 \tilde h_2 - \tilde h_2 k_2 -\tilde h_1 \tilde k_1}{\tilde h_3},
%\]
%so that, for $i=1,2$,
%\[
%2 \frac{\p }{\p \tilde k_i} (|\tilde \mu|^2_{H^1(\tilde T)}) = - \frac{\tilde h_i}{\tilde k_i^2} - \frac{\tilde h_3}{\tilde k_3^2} \frac{\p \tilde k_3}{\p \tilde k_i}  = - \frac{\tilde h_i}{\tilde k_i^2} + \frac{\tilde h_3}{\tilde k_3^2} \frac{\tilde h_i}{\tilde h_3}
%=  - \frac{\tilde h_i}{\tilde k_i^2} + \frac{\tilde h_i}{\tilde k_3^2}.
%\]
%Therefore critical points of the mapping $(\tilde k_2,\tilde k_1)\to 2 |\tilde \mu|_{H^1(\tilde T)}^2$ satisfy
%\[
%\tilde k_1  = \tilde k_2 = \tilde k_3=:\tilde k,
%\]
%which is a characterization of the incenter.  A calculation 
%shows that the  Hessian matrix of the mapping $(\tilde k_2,\tilde k_1)\to 2 |\tilde \mu|_{H^1(\tilde T)}^2$
%is positive definite, so the choice of the incenter 
%minimizes $|\tilde \mu|_{H^1(\tilde T)}^2$.  The value of $|\tilde \mu|_{H^1(\tilde T)}^2$
%in this case is
%\[
%|\tilde \mu|_{H^1(\tilde T)}^2 = \frac12 \frac1{\tilde k} |\p \tilde T| = \frac12 \frac{|\p \tilde T|}{2 |\tilde T|} |\p \tilde T| \sim \tilde \asp^{-1}.
%\]
%Note this calculation works on arbitrary $T\in \calT_h$.
%\end{remark}
%}

%%%%%%%%%%%%%%%%%%%%%%%%%%%%%%%%%%%%%%%%%%%%%%%%%%%%
%%%%%%%%%%%%%%%%%%%%%%%%%%%%%%%%%%%%%%%%%%%%%%%%%%%%
\subsection{Summary of Stability Results}
We summarize the stability results in the following theorem.
Combining Theorem \ref{thm:StdStability},
Theorem \ref{thm:ApelResult},
and Lemmas \ref{lem:StabBary}--\ref{lem:StabInc} 
yields the following result.
\begin{theorem}\label{thm:MainStabey}
Suppose $\calT_h$ satisfies a large angle condition.
Suppose further that isotropic patches, boundary layer patches
or corner patches are used (cf.~Section \ref{sec:Apel-summary}).
Let $\calT^{ct}_h$ denote the Clough-Tocher refinement
with respect to either the barycenter or incenter of each $T\in \calT_h$.
Then the inf-sup condition \eqref{eqn:InfSupStatement}
is satisfied, where {$\beta \ge C\min_{T\in \calT_h} \varrho_T^{-1}$}.

\end{theorem}

\section{Numerical Experiments}\label{sec-numerics}
In this section we numerically investigate the theoretical results from the previous sections and explore the performance of the traditional barycenter refined meshes versus incenter refinement. All calculations are performed using the finite element software FEniCS \cite{LNW12}. The associated code can be found on GitHub at https://github.com/mschneier91/anisotropic-SV. 

\subsection{Barycenter vs Incenter Aspect Ratio, Inf-Sup Constant, and Scaling}
For the first numerical experiment we examine the aspect ratio, inf-sup constant, and scaling between these quantities for the different refinement methods. We begin with an initial $ 2 \times 2 $ mesh on ${\Omega = (0,1)^2}$ %[0,1] \times [0,1]$ 
and perform a repeated barycenter or incenter refinement. The aspect ratio on the barycenter refined mesh, ${\varrho_{bary}}$, and incenter refined mesh, ${\varrho_{inc}}$, are defined as the maximum aspect ratio over all mesh cells. In practice we do not recommend this mesh refinement strategy as the error {of} a solution would plateau due to mesh edges not being refined (see \cite{FMSW21} for a hierarchical approach that is convergent). However, this refinement strategy allows for easy numerical inspection of the theoretical results proven in Section \ref{sec-prelims} and Section \ref{sec-stability}. 

We see in Table \ref{tbl:numex1} and Table \ref{tbl:numex2} that ${\varrho_{inc}}$ is {smaller} than ${\varrho_{bary}}$ at all refinement levels. Additionally, ${\varrho_{inc}}$ increases by a factor of $2$ at each refinement level whereas ${\varrho_{bary}}$ increases by a factor of $3$. These numerical results align with the bounds proven in Lemma \ref{lem:IncAsp} and Lemma \ref{lem:BaryAsp}. 

It is also shown in Table \ref{tbl:numex1} and Table \ref{tbl:numex2} that the inf-sup constant for both refinement strategies scales linearly with {${\varrho}^{-1}$}. This results in a larger inf-sup constant for for incenter refinement compared to barycenter refinement due to the smaller aspect ratio resulting from using incenter refinement. This result conforms with the theoretical scaling proven in Theorem \ref{thm:MainStabey}.

\begin{table}[h]
\caption{Inf-Sup Constant Aspect Ratio and Rate Dependence for Barycenter Refined Mesh.}
\begin{center}
 \begin{tabular}{c | c | c | c} 
 refinement level &  {$\beta_{bary}$}  & ${\varrho_{bary}}$ & rate \\ [0.5ex] 
 \hline
 1 & .26301 & 12.32 & - \\ 
 \hline
 2 & .18898 & 36.11 & .30749 \\
 \hline
 3 & .06402 & 108.03& .98777 \\
 \hline
 4 & .02137 & 324.01 & .99862 \\
 \hline
 5 & .00713 & 972.00 & .99985 \\ 
 \hline
 6 & .00238 & 2916.00 & .99998 \\ 
\end{tabular}
\end{center}
\label{tbl:numex1}
\end{table}

\begin{table}[h]
\caption{Inf-Sup Constant Aspect Ratio and Rate Dependence for Incenter Refined Mesh.}
\begin{center}
 \begin{tabular}{c | c | c | c} 
 refinement level &  {$\beta_{inc}$}  & {$\varrho_{inc}$} & rate \\ [0.5ex] 
 \hline
 1 & .27880 & 10.05 & - \\ 
 \hline
 2 & .27590 & 20.30 & .01493 \\
 \hline
 3 & .13861 & 40.71& .98959 \\
 \hline
 4 & .06939 & 81.47 & .99739 \\
 \hline
 5 & .03471 & 162.96 & .99934 \\ 
 \hline
 6 & .01735 & 325.94 & .99984 \\ 
\end{tabular}
\end{center}
\label{tbl:numex2}
\end{table}

% \begin{table}[h]\label{tbl:numex1}
% \caption{Aspect Ratio Comparison for Barycenter vs Incenter Refined Mesh}
% \begin{center}
%  \begin{tabular}{c | c | c } 
%  refinement level &  $\rho_{inc}^{ct}$ & $\rho_{bary}^{ct}$ \\ [0.5ex] 
%  \hline
%  1 & .08114 & .09946 \\ 
%  \hline
%  2  & .02769 & .049246 \\
%  \hline
%  3  & .00926 & 02456 \\
%  \hline
%  4  & .00308 & .01227 \\
%  \hline
%  5  & .00103 & .00614  \\ 
%  \hline
%  6  & .00034 & 0.00307 \\ 
% \end{tabular}
% \end{center}
% \end{table}

\subsection{Convergence of Barycenter vs Incenter Refinement}
For the second numerical experiment we consider the test problem for the steady state Stokes equation  used in \cite{AK20}. We take the domain {$\Omega =(0,1)^2$} % [0,1] \times [0,1]$ 
and the exact solution
\begin{equation*}
\begin{aligned}
u &= \left(\frac{\partial \xi}{\partial x_2}, -\frac{\partial \xi}{\partial x_1}\right),\qquad
p  = \exp\left(- \frac{x_1}{\epsilon}\right),
\end{aligned}
\end{equation*}
where the stream function is defined as
\begin{equation*}
\xi = x_1^2 (1-x_1)^2 x_2^2(1-x_2)^2\exp\left(- \frac{x_1}{\epsilon}\right).
\end{equation*}
This exact solution is characterized by the fact that the velocity and pressure have an exponential boundary layer of width $\mathcal{O}(\epsilon)$ near $x_1 = 0$. For our computations the parent grid will be the same Shishkin-type mesh used in \cite{AK20}. Letting $N \geq 2$ and $\tau \in (0,1)$ we generate a grid of points
\begin{equation*}
\begin{aligned}
x_1^i &= \begin{cases}
			i \frac{2 \tau}{N}, & 0 \leq i \leq \frac{N}{2}, \ i \in \mathbb{N},\\
            \tau + \left(i - \frac{N}{2}\right)2\frac{(1-\tau)}{N}, & \frac{N}{2} < i \leq N, \ i \in \mathbb{N},
		 \end{cases}
		 \\
    x_2^j &= \frac{j}{N}, 0 \leq j \leq N, \ j \in \mathbb{N},
\end{aligned}
\end{equation*}
and then connect the gird points with edges to obtain a rectangular mesh. Each rectangle is then subdivided into two triangles yielding a triangulation of $\Omega$ with $n = 2 N^2$ elements and an aspect ratio of
\begin{equation*}
\rho = \frac{\sqrt{1 + 4\tau^2}}{1 + 2 \tau - \sqrt{1 + 4\tau^2}}. 
\end{equation*}
An example of this mesh can be seen in Fig. \ref{fig:shisken}.

\begin{figure}[h]
\centering
\includegraphics[scale=.5]{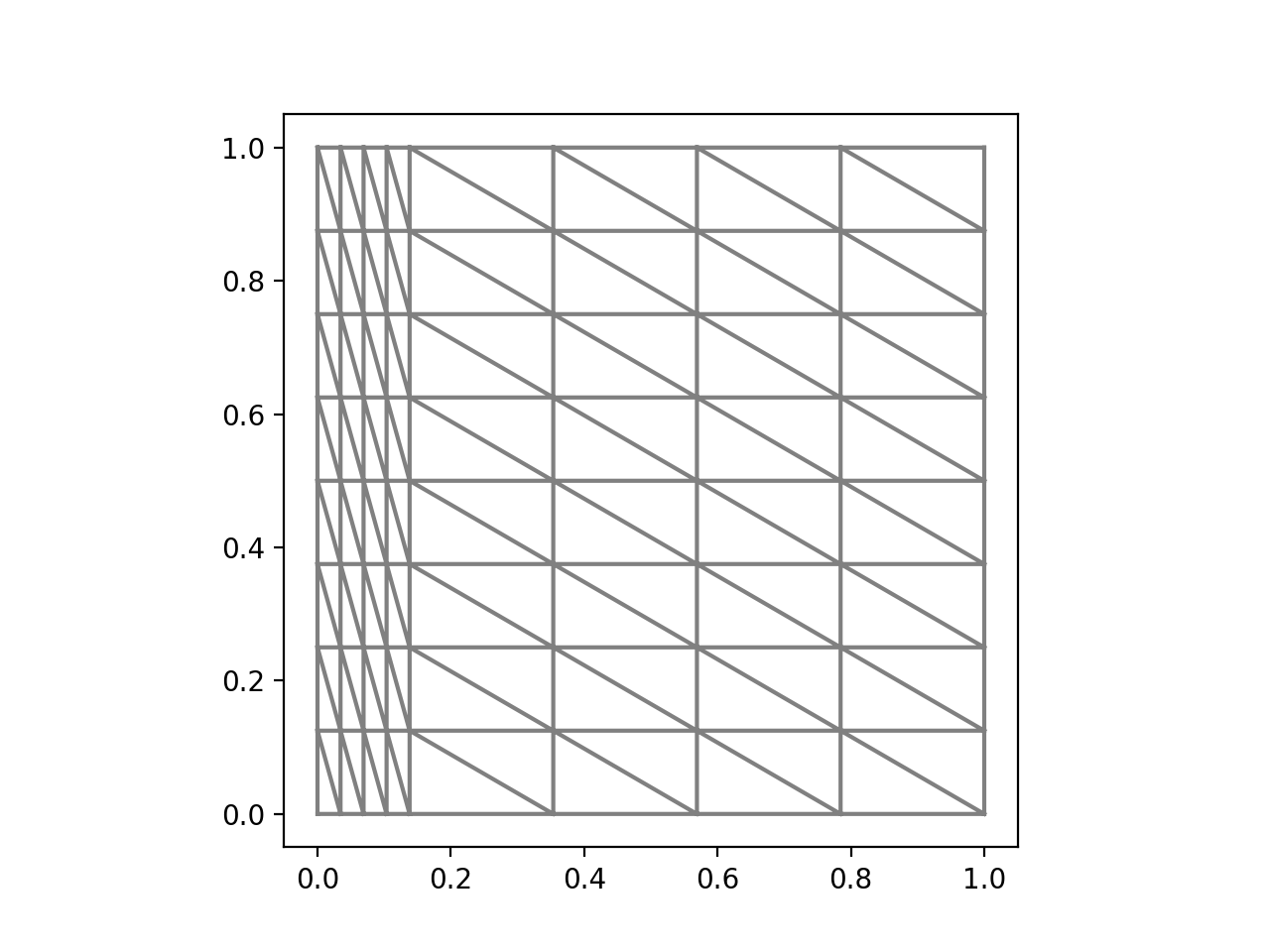}
\caption{Shishkin type mesh with $N =8$, $\epsilon = .01$, and $\tau = 3\epsilon|\log \epsilon|$.}
\label{fig:shisken}
\end{figure}

For this numerical experiment we compare a single barycenter and incenter refinement with $\epsilon = .01$, $\tau = 3\epsilon|\log \epsilon|$, and for varying values of $N$. This results in aspect ratios of ${\varrho_{inc}} \approx 18$ and ${\varrho_{bary}} \approx 24$. We see in Fig. \ref{fig:shisken_errors_velocity} the difference in velocity errors is negligible between the two refinement strategies. However, we see in Fig. \ref{fig:shisken_errors_pressure} there is a small, but noticeable improvement in the pressure error when the incenter refinement is used.

\begin{figure}[h]
\centering
\includegraphics[scale=.4]{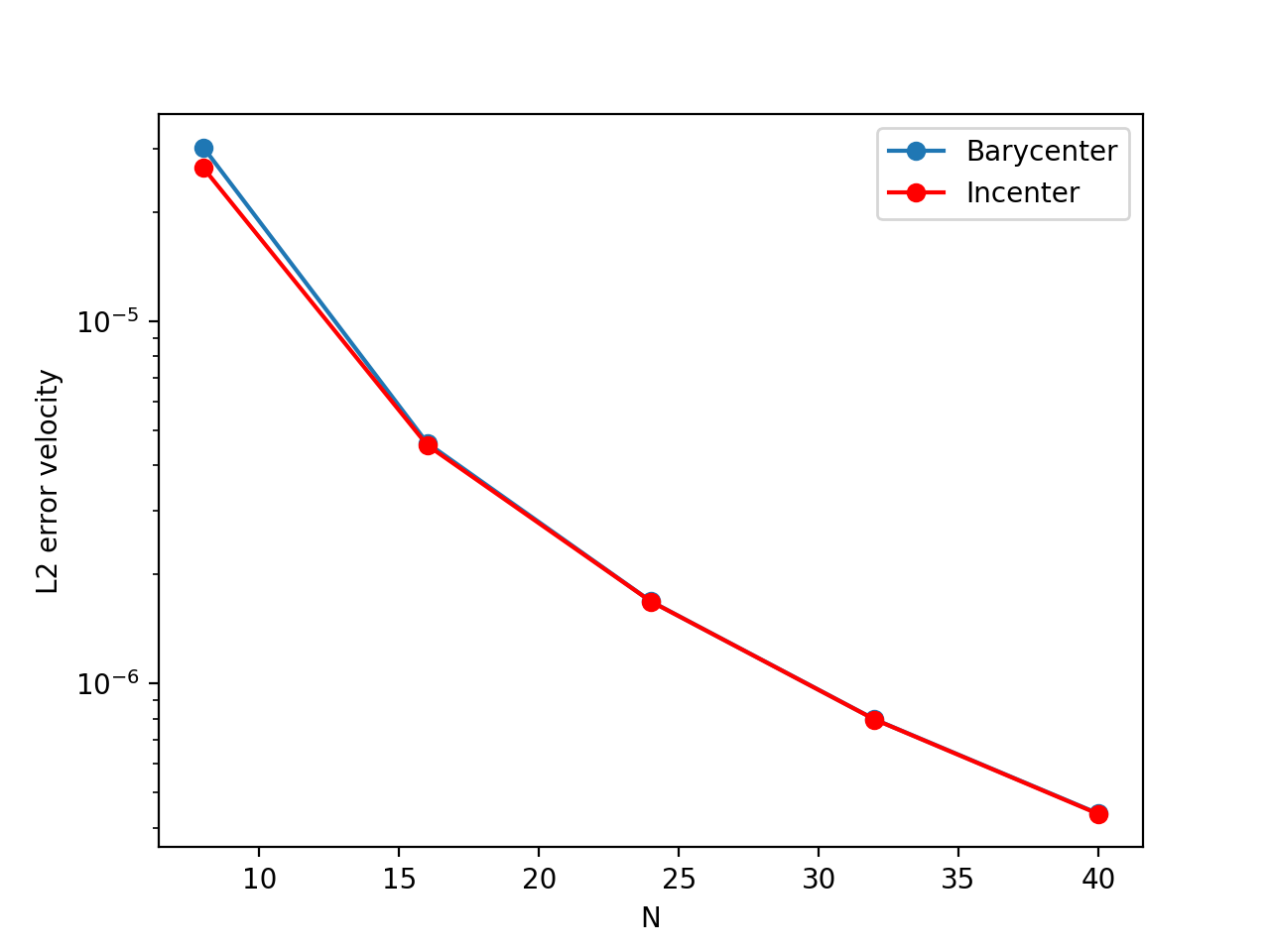}
\includegraphics[scale=.4]{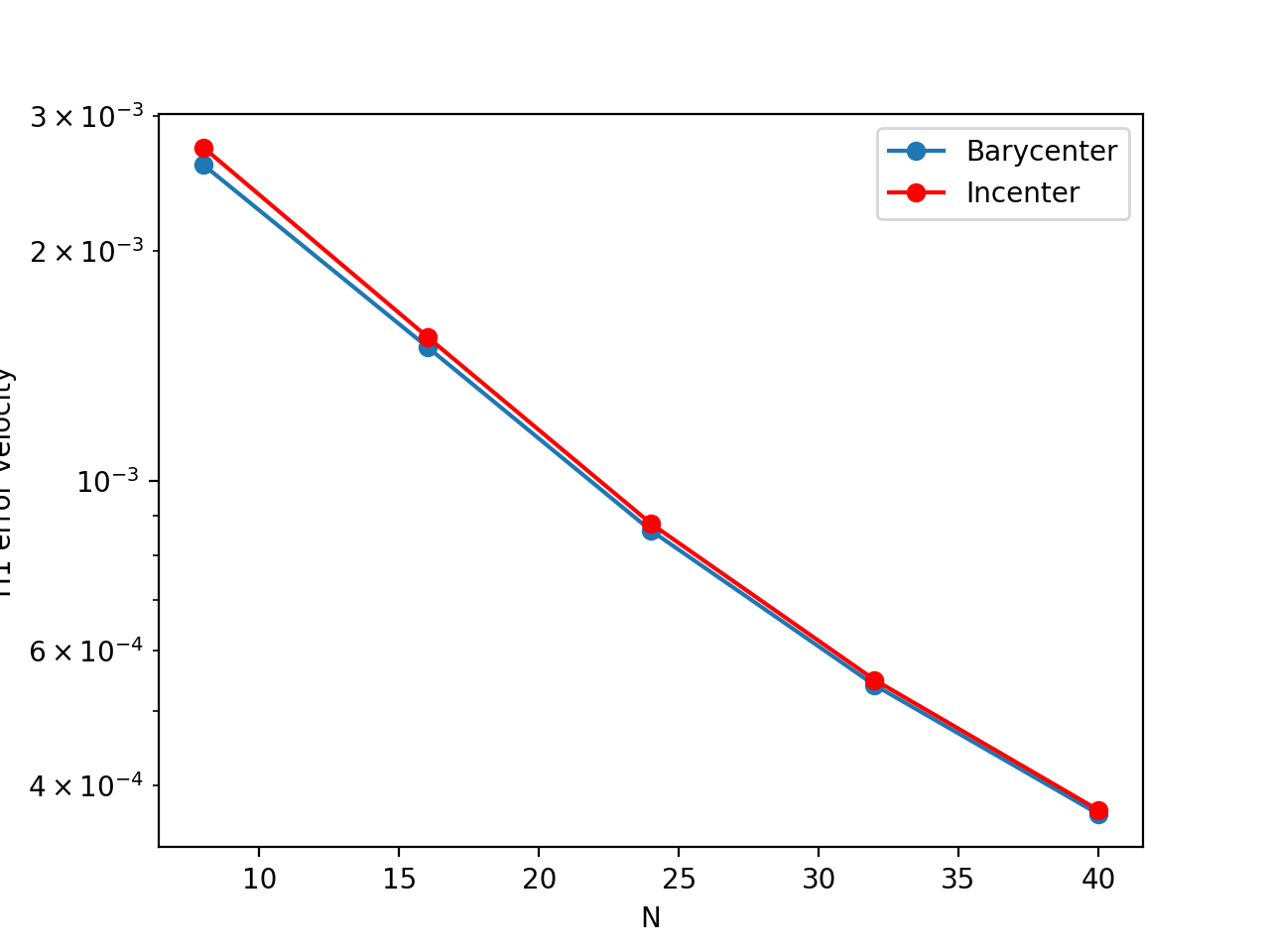}
\caption{Comparison  of $L^2$ velocity error (left) and $H^1$ error (right) for incenter versus barycenter refinement.}
\label{fig:shisken_errors_velocity}
\end{figure}

\begin{figure}[h]
\centering
\includegraphics[scale=.4]{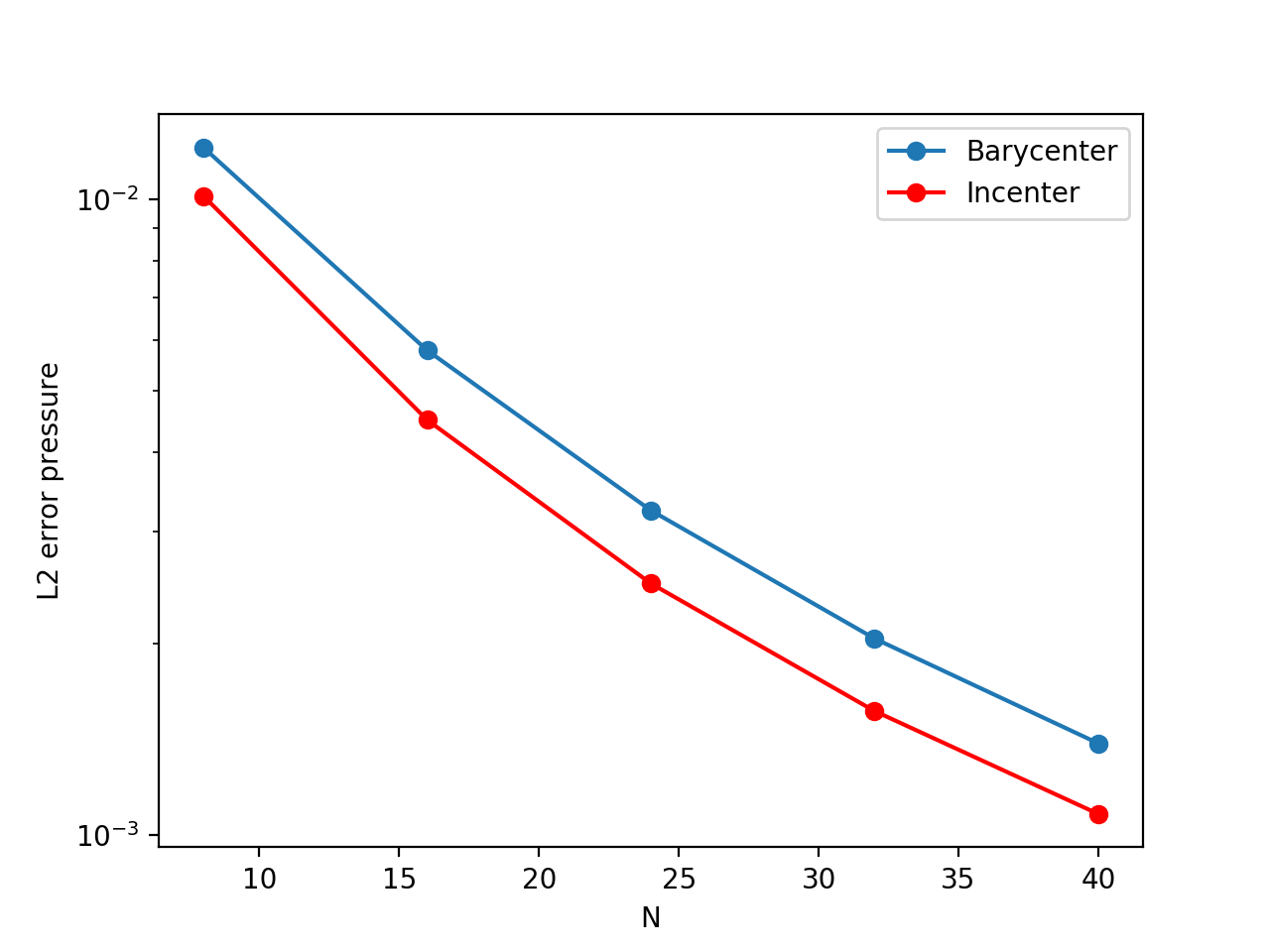}
\caption{Comparison  of $L^2$ pressure error  for incenter versus barycenter refinement.}
\label{fig:shisken_errors_pressure}
\end{figure}

\bibliographystyle{plain}
\bibliography{references}

\begin{thebibliography}{10}

\bibitem{Austin_etal04}
Travis~M. A., T.~A. Manteuffel, and Steve McCormick.
\newblock A robust multilevel approach for minimizing {$\bold H({\rm
  div})$}-dominated functionals in an {$\bold H^1$}-conforming finite element
  space.
\newblock {\em Numer. Linear Algebra Appl.}, 11(2-3):115--140, 2004.

\bibitem{AcostaEtAl}
G.~Acota, T.~Apel, R.~G. Dur\'an, and A.~L. Lombardi.
\newblock Error estimates for {R}aviart-{T}homas interpolation of any order on
  anisotropic tetrahedra.
\newblock {\em Mathematics of Computation}, 80(273):141--163, 2011.

\bibitem{Ahmed_etal18}
Naveed Ahmed, Alexander Linke, and Christian Merdon.
\newblock Towards pressure-robust mixed methods for the incompressible
  {N}avier-{S}tokes equations.
\newblock {\em Comput. Methods Appl. Math.}, 18(3):353--372, 2018.

\bibitem{ABW15}
M.~Ainsworth, G.~R. Barrenechea, and A.~Wachtel.
\newblock Stabilization of high aspect ratio mixed finite elements for
  incompressible flow.
\newblock {\em SIAM Journal on Numerical Analysis}, 53(2):1107--1120, 2015.

\bibitem{AK20}
T.~Apel and V.~Kempf.
\newblock {B}rezzi--{D}ouglas--{M}arini interpolation of any order on
  anisotropic triangles and tetrahedra.
\newblock {\em SIAM Journal on Numerical Analysis}, 58(3):1696--1718, 2020.

\bibitem{AKLM21}
T.~Apel, V.~Kempf, A.~Linke, and C.~Merdon.
\newblock {A nonconforming pressure-robust finite element method for the
  {S}tokes equations on anisotropic meshes}.
\newblock {\em IMA Journal of Numerical Analysis}, 01 2021.
\newblock draa097.

\bibitem{ApelNicaise04}
T.~Apel and S.~Nicaise.
\newblock The inf-sup condition for low order elements on anisotropic meshes.
\newblock {\em Calcolo}, 41:89--113, 2004.

\bibitem{ANS01}
T.~Apel, S.~Nicaise, and J.~Sch\"{o}berl.
\newblock Crouzeix-{R}aviart type finite elements on anisotropic meshes.
\newblock {\em Numer. Math.}, 89(2):193--223, 2001.

\bibitem{ArnoldQin92}
D.~N. Arnold and J.~Qin.
\newblock Quadratic velocity/linear pressure {S}tokes elements.
\newblock In {\em Advances in Computer Methods for Partial Differential
  Equations VII}, pages 28--34. IMACS, 1992.

\bibitem{BW19}
G.~R. Barrenechea and A.~Wachtel.
\newblock {The inf-sup stability of the lowest order Taylor–Hood pair on
  affine anisotropic meshes}.
\newblock {\em IMA Journal of Numerical Analysis}, 40(4):2377--2398, 07 2019.

\bibitem{ART15}
M.~Akbas Belenli, L.~G. Rebholz, and F.~Tone.
\newblock A note on the importance of mass conservation in long-time stability
  of {N}avier–{S}tokes simulations using finite elements.
\newblock {\em Applied Mathematics Letters}, 45:98--102, 2015.

\bibitem{BernardiRaugel85}
C.~Bernardi and G.~Raugel.
\newblock Analysis of some finite elements for the {S}tokes problem.
\newblock {\em Mathematics of Computation}, 44(169):71--79, 1985.

\bibitem{BoffiEtAl}
B.~Boffi, R.~Brezzi, L.F. Demkowicz, R.G. Dur\'an, R.S. Falk, and M.~Fortin.
\newblock Mixed finite elements, compatibility conditions, and applications.
\newblock {\em Lectures given at the C.I.M.E. Summer School held in Cetraro,
  June 26–July 1, 2006}, 44(169):71--79, 2008.

\bibitem{Brennecke_etal15}
C.~Brennecke, A.~Linke, C.~Merdon, and J.~Sch\"{o}berl.
\newblock Optimal and pressure-independent {$L^2$} velocity error estimates for
  a modified {C}rouzeix-{R}aviart {S}tokes element with {BDM} reconstructions.
\newblock {\em J. Comput. Math.}, 33(2):191--208, 2015.

\bibitem{Buffa_etal11}
A.~Buffa, C.~de~Falco, and G.~Sangalli.
\newblock Iso{G}eometric {A}nalysis: stable elements for the 2{D} {S}tokes
  equation.
\newblock {\em Internat. J. Numer. Methods Fluids}, 65(11-12):1407--1422, 2011.

\bibitem{Cockburn_etal05}
B.~Cockburn, G.~Kanschat, and D.~Schotzau.
\newblock A locally conservative {LDG} method for the incompressible
  {N}avier-{S}tokes equations.
\newblock {\em Math. Comp.}, 74(251):1067--1095, 2005.

\bibitem{DS21}
V.~DeCaria and M.~Schneier.
\newblock An embedded variable step imex scheme for the incompressible
  {N}avier–{S}tokes equations.
\newblock {\em Computer Methods in Applied Mechanics and Engineering},
  376:113661, 2021.

\bibitem{FalkNeilan13}
R.~S. Falk and M.~Neilan.
\newblock Stokes complexes and the construction of stable finite elements with
  pointwise mass conservation.
\newblock {\em SIAM J. Numer. Anal.}, 51(2):1308--1326, 2013.

\bibitem{FMSW21}
P.~E. Farrell, L.~Mitchell, L.~Ridgway Scott, and F.~Wechsung.
\newblock A {Reynolds-robust} preconditioner for the {Scott-Vogelius}
  discretization of the stationary incompressible {Navier-Stokes} equations.
\newblock {\em The SMAI journal of computational mathematics}, 7:75--96, 2021.

\bibitem{Guzman_etal20}
J.~Guzm\'{a}n, A.~Lischke, and M.~Neilan.
\newblock Exact sequences on {P}owell-{S}abin splits.
\newblock {\em Calcolo}, 57(2):Paper No. 13, 25, 2020.

\bibitem{GuzmanNeilan14}
J.~Guzm\'{a}n and M.~Neilan.
\newblock Conforming and divergence-free {S}tokes elements on general
  triangular meshes.
\newblock {\em Math. Comp.}, 83(285):15--36, 2014.

\bibitem{GuzmanNeilan18}
J.~Guzm\'an and M.~Neilan.
\newblock Inf-sup stable finite elements on barycentric refinements producing
  divergence--free approximations in arbitrary dimensions.
\newblock {\em SIAM Journal on Numerical Analysis}, 56(5):2826--2844, 2018.

\bibitem{JNN18}
V.~John, P.~Knobloch, and J.~Novo.
\newblock Finite elements for scalar convection-dominated equations and
  incompressible flow problems: a never ending story?
\newblock {\em Computing and Visualization in Science}, 19:47--63, 2018.

\bibitem{SIAM_DivFree17}
V.~John, A.~Linke, C.~Merdon, M.~Neilan, and L.~G. Rebholz.
\newblock On the divergence constraint in mixed finite element methods for
  incompressible flows.
\newblock {\em SIAM Rev.}, 59(3):492--544, 2017.

\bibitem{KreuzerZanotti20}
C.~Kreuzer and P.~Zanotti.
\newblock Quasi-optimal and pressure-robust discretizations of the {S}tokes
  equations by new augmented {L}agrangian formulations.
\newblock {\em IMA J. Numer. Anal.}, 40(4):2553--2583, 2020.

\bibitem{Lederer_etal17}
P.~L. Lederer, A.~Linke, C.~Merdon, and J.~Sch\"{o}berl.
\newblock Divergence-free reconstruction operators for pressure-robust {S}tokes
  discretizations with continuous pressure finite elements.
\newblock {\em SIAM J. Numer. Anal.}, 55(3):1291--1314, 2017.

\bibitem{Linke14}
A.~Linke.
\newblock On the role of the {H}elmholtz decomposition in mixed methods for
  incompressible flows and a new variational crime.
\newblock {\em Comput. Methods Appl. Mech. Engrg.}, 268:782--800, 2014.

\bibitem{Linke_etal16}
A.~Linke, G.~Matthies, and L.~Tobiska.
\newblock Robust arbitrary order mixed finite element methods for the
  incompressible {S}tokes equations with pressure independent velocity errors.
\newblock {\em ESAIM Math. Model. Numer. Anal.}, 50(1):289--309, 2016.

\bibitem{LinkeMerdon16}
A.~Linke and C.~Merdon.
\newblock Pressure-robustness and discrete {H}elmholtz projectors in mixed
  finite element methods for the incompressible {N}avier-{S}tokes equations.
\newblock {\em Comput. Methods Appl. Mech. Engrg.}, 311:304--326, 2016.

\bibitem{LNW12}
A.~Logg, K.~Mardal, and G.~Wells.
\newblock {\em Automated solution of differential equations by the finite
  element method: The FEniCS book}, volume~84.
\newblock Springer Science \& Business Media, 2012.

\bibitem{Neilan20}
M.~Neilan.
\newblock The {S}tokes complex: a review of exactly divergence-free finite
  element pairs for incompressible flows.
\newblock In {\em 75 years of mathematics of computation}, volume 754 of {\em
  Contemp. Math.}, pages 141--158. Amer. Math. Soc., [Providence], RI, [2020]
  \copyright 2020.

\bibitem{SS98}
D.~Schotzau and C.~Schwab.
\newblock Mixed hp-fem on anisotropic meshes.
\newblock {\em Mathematical Models and Methods in Applied Sciences},
  08(05):787--820, 1998.

\bibitem{SSR99}
D.~Schotzau, C.~Schwab, and R.~Stenberg.
\newblock Mixed hp-fem on anisotropic meshes {II}: Hanging nodes and tensor
  products of boundary layer meshes.
\newblock {\em Numerische Mathematik}, 83:667–697, 1999.

\bibitem{ScottVogelius85}
L.~R. Scott and M.~Vogelius.
\newblock Norm estimates for a maximal right inverse of the divergence operator
  in spaces of piecewise polynomials.
\newblock {\em RAIRO Mod\'{e}l. Math. Anal. Num\'{e}r.}, 19(1):111--143, 1985.

\bibitem{VerfurthZanotti19}
R.~Verf\"{u}rth and P.~Zanotti.
\newblock A quasi-optimal {C}rouzeix-{R}aviart discretization of the {S}tokes
  equations.
\newblock {\em SIAM J. Numer. Anal.}, 57(3):1082--1099, 2019.

\bibitem{Zhang05}
S.~Zhang.
\newblock A new family of stable mixed finite elements for the 3{D} {S}tokes
  equations.
\newblock {\em Math. Comp.}, 74(250):543--554, 2005.

\end{thebibliography}

\newpage

\appendix

\section{Proof of Lemma \ref{mapisgood}}

\begin{proof}
First, we may assume that $\delta < \frac{\pi}{3}.$

We will use the same notation for edges and vertices as in Section \ref{sec-prelims}, in particular preserving the ordering of side lengths.

Define $t_1$ to be the unit vector along $e_1$, and $t_2$ to be the unit vector along $e_2.$ That is, $t_1=\frac{z_2-z_3}{h_1}, t_2 = \frac{z_1-z_3}{h_2}$. Then, let $A$ have columns $t_1,t_2.$ Let $b = z_3.$ Then, the affine map $A\tilde{x} + b$ maps $\tilde{T}$ onto $T.$

It is clear that as each entry in $A$ is bounded above by 1 as the columns are unit vectors. Thus $\|A\|\leq 2$ and $\|adj(A)\|\leq 2,$ giving us $\|A^{-1}\|\leq \frac{2}{|det(A)|}.$ 

It is well known that $|det(A)|$ is the area of the parallelogram formed by the vectors $t_1$ and $t_2.$ We then have the formula $$|det(A)|=|t_1||t_2|\sin{\alpha_3}=\sin{\alpha_3}$$ As $\alpha_3\in[\frac{\pi}{3},\pi-\delta],$ $\sin{\alpha_3}\geq\sin{\delta}$, and $\|A^{-1}\| \leq \frac{2}{\sin{\delta}}=C(\delta).$\hfill
\end{proof}

%%%%%%%%%%%%%%%%%%%%%%%%%%%%%%%%%%%%%%%%%%%%%%%%
%%%%%%%%%%%%%%%%%%%%%%%%%%%%%%%%%%%%%%%%%%%%%%%%
\section{Proof of Lemma \ref{lem:BDMScaling}}

We denote by $\tilde F:\hat T\to \tilde T$ the affine mapping 
given by
\begin{align*}
\tilde F(\hat x) = \begin{pmatrix}
h_1 & 0\\
0 & h_2
\end{pmatrix}.
\end{align*}
For $ \tilde \bv\in \bpol_1(\tilde T)$,
let $\hat \bv\in \bpol_1(\hat T)$ be given via the Piola transform
\begin{align}\label{eqn:tildePiola}
 \tilde \bv(\tilde x) = \frac1{\det(D\tilde F)}D\tilde F \hat \bv(\hat x)=
\begin{pmatrix}
h_2^{-1} \hat v_1(\hat x)\\
h_1^{-1} \hat v_2(\hat x)
\end{pmatrix},\qquad  \tilde x = \tilde F(\hat x).
\end{align}
\begin{proof}
Let $\hat \bv\in \bpol_1(\hat T)$ be defined by \eqref{eqn:tildePiola}.
The  Brezzi-Douglas-Marini DOFs and equivalence of norms yields
\begin{align*}
\|\hat \bv\|^2 
\le C \sum_{i=1}^3 \|\hat \bv\cdot \hat \bn_i\|_{L^2(\hat e_i)}^2
\end{align*}
for any norm $\|\cdot\|$ on $\bpol_1(\hat T)$.
Using the identity $\hat \bv\cdot \hat \bn_i(\hat x) = (h_i/|\hat e_i|) (\bv\cdot \bn_i)(x)$, we have, by a change of variables,
\begin{align*}
\|\hat \bv\|^2
&\le C \sum_{i=1}^3 h_i \| \bv \cdot  \bn_i\|_{L^2( e_i)}^2.
\end{align*}

Furthermore, by the chain rule
\begin{align*}
 \nab  \bv( x) 
&= \frac1{\det(D\tilde F)}D\tilde F \hat \nabla \hat \bv(\hat x)(D\tilde F)^{-1}
 = \frac1{h_1 h_2} \begin{pmatrix}
\frac{\p \hat v_1}{\p \hat x_1} & \frac{h_1}{h_2} \frac{\p \hat v_1}{\p \hat x_2}\\
\frac{h_2}{h_1} \frac{\p \hat v_2}{\p \hat x_1} & \frac{\p \hat v_2}{\p \hat x_2}.
\end{pmatrix} 
\end{align*}
Therefore,
\begin{align*}
\| \nab  \bv\|_{L^2(\tilde T)}^2
&\le 2 |\tilde T| (h_1 h_2)^{-2} \max\{\frac{h_1}{h_2},\frac{h_2}{h_1}\} \|\hat \nab \hat \bv\|_{L^2(\hat T)}^2\\
&=  (h_1 h_2)^{-1}\varrho \|\hat \nab \hat \bv\|_{L^2(\hat T)}^2\\
&\le C |\tilde T|^{-1}\varrho \sum_{i=1}^3 k_i \|\bv\cdot \bn_i\|_{L^2(e_i)}^2.
\end{align*}
We also have
\begin{align*}
\| \bv\|_{L^\infty(\tilde T)}^2  = \max\{h_2^{-2},h_1^{-2}\} \|\hat \bv\|_{L^\infty(\hat T)}^2
\le  C |\tilde T|^{-1}\varrho \sum_{i=1}^3 h_i \| \bv \cdot  \bn\|_{L^2( e_i)}^2.
\end{align*}\hfill
\end{proof}

%%%%%%%%%%%%%%%%%%%%%%%%%%%%%%%%%%%%%%%%%%%%%%%%
%%%%%%%%%%%%%%%%%%%%%%%%%%%%%%%%%%%%%%%%%%%%%%%%
%\subsection{Proof of Lemma \ref{lem:HatScale}}
%\begin{proof}
%The function $\tilde \mu$ satisfies  $\tilde \nab \tilde \mu|_{\tilde K_i} = -\frac1{\tilde k_i} \tilde \bn_i$.  Therefore,
%\begin{align*}
%|\tilde \mu|_{H^1(\tilde T)}^2
%& = \sum_{i=1}^3 |\tilde K_i| \tilde k_i^{-2}
% = \sum_{i=1}^3 \frac{\tilde h_i}{\tilde k_i}.
% \end{align*} 
% \end{proof}
 
 \subsection{Proof of Lemma \ref{lem:Tracey}}
  \begin{proof}
 We have 
 \begin{align*}
 k_i \|q\|_{L^2(e_i)}^2 \le h_i k_i \|q\|_{L^\infty(e_i)}^2 \le h_i k_i \|q\|_{L^\infty(K_i)}^2.
 \end{align*}
 Therefore by standard scaling,
 \begin{align*}
  k_i \|q\|_{L^2(e_i)}^2 \le C h_i k_i |K_i|^{-1} \|q\|_{L^2(K_i)}^2 \le C \|q\|_{L^2(K_i)}^2.
 \end{align*}\hfill
  \end{proof}

\end{document}